\documentclass[11pt, oneside]{amsart}  	% use "amsart" instead of "article" for AMSLaTeX format
\usepackage{geometry}                		% See geometry.pdf to learn the layout options. There are lots.
\usepackage{amssymb}
\usepackage{comment}
\usepackage[dvipdfmx]{xcolor}

\usepackage{amsmath} % for \tag* macro

\usepackage{graphicx}

\theoremstyle{theorem}
\newtheorem {theo}{Theorem}[section]
\newtheorem*{theoA*}{Theorem 1}
\newtheorem*{theoB*}{Theorem 2}
\newtheorem {lem}[theo]{Lemma}
\newtheorem*{lem*}{Lemma}
\newtheorem {prop}[theo]{Proposition}
\newtheorem*{prop*}{Proposition}
\newtheorem {cor}[theo]{Corollary}
\newtheorem*{cor*}{Corollary}

\newtheorem*{conjecture*}{Conjecture}
\theoremstyle{definition}
\newtheorem {defi}[theo]{Definition}
\newtheorem*{defi*}{Definition}

\newtheorem*{nota*}{Notation}
\theoremstyle{remark}
\newtheorem {rem}[theo]{Remark}
\newtheorem*{rem*}{Remark}

\newtheorem*{warning*}{Warning}

\newtheorem*{warnings*}{Warnings}
\newtheorem {convention}[theo]{Convention}
\newtheorem*{convention*}{Convention}
\newtheorem {exemple}[theo]{Example}
\newtheorem*{exemple*}{Example}

\newtheorem*{exemples*}{Examples}

\newtheorem*{question*}{Question}

\newtheorem*{questions*}{Questions}

\newtheorem*{fact*}{Fact}

\newtheorem*{acknowledgments}{Acknowledgments}

\newcommand{\1}{\mathbf{1}}
\newcommand*{\Searrow}{\rotatebox[origin=c]{-45}{\(\Longrightarrow\)}}
\newcommand{\nbpt}{18}
\newcommand{\figtotext}[3]{\begin{array}{c}\includegraphics[scale=0.7]{#3}\end{array}}
\newcommand{\Ov}{\figtotext{\nbpt}{\nbpt}{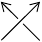}}
\newcommand{\und}{\figtotext{\nbpt}{\nbpt}{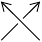}}
\newcommand{\smoo}{\figtotext{\nbpt}{\nbpt}{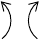}}

\newcommand{\ps}{p\#}%{\hspace{-1.3ex}\begin{array}{c}\vspace{-0.45ex} \includegraphics[scale=0.16]{gam.pdf}\end{array}\!\!\!}
\newcommand{\bps}{\textrm{band-}p\#}%{\textcolor{red}{\textrm{band-}\hspace{-1ex}\begin{array}{c}\hspace{-0.15ex}\upharpoonleft\hspace{-0.5ex}\upharpoonright\\[-1.15em]\rightleftharpoons\end{array}\hspace{-1ex}}}
\newcommand{\shp}{\#}%{\hspace{-1.3ex}\begin{array}{c}\vspace{-0.45ex} \includegraphics[scale=0.16]{sh.pdf}\end{array}\!\!\!}
\title[Links up to clasp-pass moves]{The classification of links up to clasp-pass moves}
\author[J.B. Meilhan]{Jean-Baptiste Meilhan} 
\address{Univ. Grenoble Alpes, CNRS, Institut Fourier, F-38000 Grenoble, France}
	 \email{jean-baptiste.meilhan@univ-grenoble-alpes.fr}
\author[A. Yasuhara]{Akira Yasuhara} 
\address{Faculty of Commerce, Waseda University, 1-6-1 Nishi-Waseda,
               Shinjuku-ku, Tokyo 169-8050, Japan}
	 \email{yasuhara@waseda.jp}
\begin{document}
\maketitle

\begin{abstract}
We give a complete classification of links up to clasp-pass moves, which coincides with Habiro's $C_3$-equivalence.
We also classify links up to band-pass and band-$\shp$ moves, which are 
versions of the usual pass- and $\shp$-move, respectively,  where each pair of parallel strands belong to the same component. 
This recovers and generalizes widely a number of partial results in the study of these local moves. 
The proofs make use of clasper theory.
\end{abstract}

\section{Introduction}
The classification of knots and links up to ambient isotopy, by means of computable and effective invariants,  remains to date an outstanding open question in low dimensional topology. 
A natural approach to this challenging problem is to seek intermediate classification results, involving coarser topological equivalence relations than isotopy. 
This was for example the idea underlying the early works of Milnor \cite{M1,M2}, who launched in the mid-fifties the study of links up to \emph{link-homotopy}. 
Here, link-homotopy is the equivalence relation generated by isotopies and the \emph{self-crossing change}, which is the local move shown on the left-hand side of Figure \ref{fig:clasp}, where both represented strands belong to the same link component.
The classification of links up to link-homotopy was only obtained in 1990 by Habegger and Lin \cite{HL}, using a variant of the so-called \emph{Milnor invariants} developed in \cite{M1,M2}.

Since the work of Milnor, a number of interesting intermediate classification problems arose in the study of knots and links, based on other local moves. 
The {\em delta move}, defined independently by Matveev and Murakami-Nakanishi in the late eighties, is the local move represented in the center of Figure~\ref{fig:clasp}. It was shown to provide a combinatorial interpretation of the link-homology relation for links, and to be classified by the linking numbers \cite{Matveev,M-N}.
\begin{figure}[!h]
\includegraphics[scale=0.82]{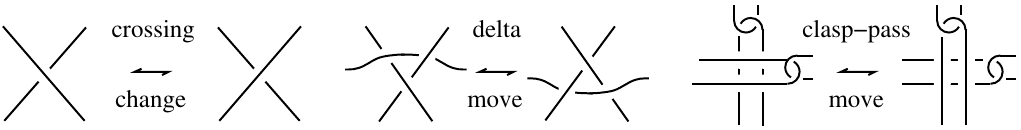} 
  \caption{A crossing change, a delta move and a clasp-pass move}  \label{fig:clasp}
\end{figure}
Habiro subsequently introduced the local move shown on the right-hand side of Figure~\ref{fig:clasp}, called \emph{clasp-pass move} \cite{Habiro0}. A clasp-pass move is realized by two delta moves.  
More generally Habiro showed that the three local moves of Figure~\ref{fig:clasp} 
arise naturally as the first three of an infinite family of local moves, called {\em $C_k$-moves} $(k\in \mathbb{N})$, 
which give finer and finer equivalence relations as $k$ increases, 
and which are deeply related to the theory of finite type invariants \cite{Habiro}.

In \cite{Habiro}, Habiro showed that two knots are equivalent up to clasp-pass moves if and only if  
their {\em Casson knot invariants} $a_2$ coincide; here, the Casson knot invariant is the second coefficient of the Alexander-Conway polynomial. For links, however, the only known clasp-pass classification results are for links with up to $3$ components and for algebraically split links, \emph{i.e.} links with vanishing linking numbers, due to Taniyama and the second author in 2002 \cite{T-Y}. 
The first main result of this paper, is the clasp-pass classification of link with any number of components (Theorem 1, stated below).
\medskip

The clasp-pass move is also intimately related to the well-known pass-moves and $\shp$-moves. 
The {\em pass-move} was defined in 1974 by Kauffman, as the local move represented on the left-hand side of Figure \ref{pass} \cite{Kauffman0}.\footnote{In \cite{Kauffman0}, Kauffman uses the terminology \lq band-pass operation\rq\, for this local move.} Kauffman showed that 
two knots are equivalent up to pass-moves if and only if their \emph{Arf invariants} coincide \cite{Kauffman}. 
\begin{figure}[!h]
\includegraphics[scale=0.85]{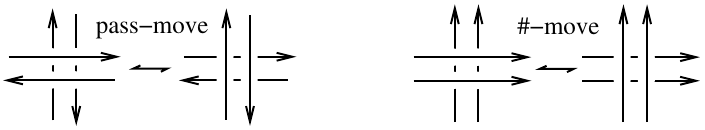} 
  \caption{A pass-move and a $\shp$-move}  \label{pass}
\end{figure}
The {\em $\shp$-move} is the local move illustrated on the right-hand side of Figure \ref{pass}. 
It was introduced in 1985 by Murakami, who showed that all knots are equivalent to the trivial knot up to $\shp$-moves \cite{Murakami}. 
Murakami and Nakanishi gave the classifications of links up to pass-moves, and also up to $\shp$-moves \cite{M-N}. 

It is easily seen that a clasp-pass move is only a special case of a pass-move.
In fact, this is more precisely a special case of a \emph{band-pass move}. 
Here a band-pass move is a pass-move where each pair of parallel strands in Figure \ref{pass} are required to belong to the same link component. Likewise, we define the notion of band-$\shp$ move. Using an observation of Murakami and Nakanishi \cite[Fig.~A.7]{M-N}, we have that a band-pass move is achieved by band-$\shp$ moves, so that these three local moves, clasp-pass, band-pass and band-$\shp$ moves, yield another coherent sequence of equivalence relations on knots and links. 

Very little is known about the band-pass and band-$\shp$ equivalence, although it has been previously studied in the literature. 
Martin gave the classification of algebraically split links up to band-pass moves \cite{Martin}. 
In 1993, Saito defined an unoriented version of these local moves, which turns out to be equivalent to the band-$\shp$ move (see Remark~\ref{rem:saito}), and gave a classification of 2-component links with even linking numbers \cite{Saito}. 
Moreover, Shibuya and the second author gave  in \cite{S-Y} classifications results for links up to \emph{self-pass} and \emph{self-$\shp$} moves,\footnote{The self-pass move was first introduced by Cervantes and Fenn in \cite{C-F}, where it is called {\em elementary pass}, and the self-$\shp$ move is due to Shibuya \cite{Shibuya0}.} where as above the prefix \lq self\rq\, stands for versions of these local moves where all strands belong to the same link component; these result refine the classification of links up to link-homotopy. Shibuya showed that two concordant links are equivalent up to self-pass  moves \cite{Shibuya}; in particular, concordance implies band-pass equivalence.\\
Our second main result, Theorem~2 below, is the classification of links up to band-pass and band-$\shp$ moves. 
\medskip 
 
 In order to state our classification theorems, we need some definitions and notation. 

 Let $D^2$ be the unit disk,  and let $D^1\subset D^2$ be a fixed oriented diameter. 
Let  $p_1,...,p_{2n}$ be $2n$ fixed points, arranged in this order along the orientation of 
$D^1\times\{0\}\subset D^2\times [0,1]$. 
An {\em $n$-component bottom tangle}  $\sigma=\sigma_1\cup\cdots\cup\sigma_n$ 
is a tangle in the cylinder $D^2\times [0,1]$ consisting of $n$ arcs $\sigma_1,...,\sigma_n$ such that 
each component $\sigma_i$ runs from $p_{2i-1}$ to  $p_{2i}$ ($1\le i\le n$). 
The term \lq bottom tangle\rq\, is due to Habiro \cite{Habiro2},
although this notion previously appeared in several places in the literature, see for example \cite{Levine}; there is a natural correspondence between bottom tangles and string links, see for example \cite[\S~13]{Habiro2}, and we shall make use of this correspondence implicitly. 
   
Given an $n$-component bottom tangle $\sigma=\sigma_1\cup\cdots\cup\sigma_n$, 
we define two closure-type operations, as follows.
On the one hand, let $\widehat{\sigma}=\widehat{\sigma_1}\cup \cdots \cup \widehat{\sigma_n}$ denote the {\em closure} of $\sigma$, which is the $n$-component link obtained by identifying the starting and ending point of each component as shown in  Figure~\ref{closure}. 
On the other hand, for any two indices $i,j$ such that $1\le i<j\le n$, let $p_{ij}(\sigma)$ denote the {\em plat closure} of the $i$th and $j$th components of $\sigma$, as illustrated in Figure~\ref{closure}; note that $p_{ij}(\sigma)$  is a knot, with orientation induced by that of the $i$th component of $\sigma$. 
 \begin{figure}[!h]
\includegraphics[scale=1]{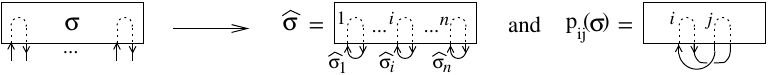}
  \caption{Two closure-type operations on $n$-component bottom tangles}\label{closure}
\end{figure}
Using these closure operations, we define 
$a_2(\sigma;i):=a_2(\widehat{\sigma_i})$ and $a_2(\sigma;ij):=a_2(p_{ij}(\sigma))$ 
where as above $a_2$ is the Casson knot invariant.
In what follows, we shall simply denote these invariants by $a_2(i)$ and $a_2(ij)$. 

Our other main ingredients are {\em Milnor invariants} of $\sigma$. These are integers $\mu_\sigma(I)$ indexed by sequences $I$ of (possibly repeated) elements of $\{1,...,n\}$. In particular, Milnor invariant $\mu(ij)$ is nothing but the linking number of the $i$th and $j$th component ($i\neq j$). See Section \ref{sec:milnor} for a review of the definition. 
\medskip

The classification of links up to clasp-pass moves reads as follows. 
\begin{theoA*}\label{main-theoremA}
Let $L$ and $L'$ be $n$-component links, and let $\sigma$ and $\sigma'$ be 
bottom tangles whose closures $\widehat{\sigma}$ and $\widehat{\sigma'}$ 
are equivalent to $L$ and $L'$ respectively. 
The links $L$ and $L'$ are clasp-pass equivalent if and only if $\sigma$ and $\sigma'$ 
share all invariants $a_2(i)~(i\in\{1,...,n\})$ and $\mu(ij)~(1\le i<j\le n)$ 
and the following system of linear equations has a solution over $\mathbb{Z}$:  
\[
\left\{
\begin{array}{l}
a_2(\sigma';ij)-a_2(\sigma;ij)\equiv \mu_{\sigma}(ij)(x_{ij}+x_{ji})\pmod2 \qquad\qquad\qquad \,\,\,\, (1\leq i<j\leq n)\\
\mu_{\sigma'}(ijk)-\mu_{\sigma}(ijk)=\mu_{\sigma}(ij)(x_{ki}-x_{kj})+\mu_{\sigma}(jk)(x_{ij}-x_{ik})
+\mu_{\sigma}(ki)(x_{jk}-x_{ji}) \\
\qquad \qquad\qquad\qquad\qquad\qquad\qquad\qquad\qquad\qquad\qquad\qquad\qquad\qquad (1\leq i<j<k\leq n)
\end{array}
\right.
\]
\end{theoA*}
In the special case where $L$ and $L'$ are algebraically split links, Theorem~1 implies that $L$ and $L'$ are clasp-pass equivalent if and only if $\sigma$ and $\sigma'$ 
share all invariants $a_2(i)$, $a_2(ij) \pmod2$ and $\mu(ijk)$ ($i<j<k$), which recovers \cite[Thm.~1.4]{T-Y}. 

\begin{theoB*}\label{main-theoremB}
Let $L$ and $L'$ be $n$-component links, and let $\sigma$ and $\sigma'$ be 
bottom tangles whose closures $\widehat{\sigma}$ and $\widehat{\sigma'}$ 
are equivalent to $L$ and $L'$ respectively. 
\begin{itemize} 
\item[(1)] The links $L$ and $L'$ are band-pass equivalent if and only if $\sigma$ and $\sigma'$ 
share all invariants $a_2(i)\pmod2~(i\in\{1,...,n\})$ and $\mu(ij)$ $(1\le i<j\le n)$ and
the following system has a solution over $\mathbb{Z}$. 
\[
\left\{
\begin{array}{l}
\mu_{\sigma'}(jiij)-\mu_{\sigma}(jiij)\equiv \mu_{\sigma}(ij)(x_{ij}+x_{ji}) \pmod2\ \qquad\qquad\qquad \,\,\,\, (1\leq i<j\leq n)\\
\mu_{\sigma'}(ijk)-\mu_{\sigma}(ijk)=\mu_{\sigma}(ij)(x_{ki}-x_{kj})+\mu_{\sigma}(jk)(x_{ij}-x_{ik})
+\mu_{\sigma}(ki)(x_{jk}-x_{ji}) \\
 \qquad \qquad\qquad\qquad\qquad\qquad\qquad\qquad\qquad\qquad\qquad\qquad\qquad\qquad (1\leq i<j<k\leq n)
\end{array}
\right.
\]
\item[(2)] The links $L$ and $L'$ are band-$\shp$ equivalent if and only if they share all invariants $\mu(ij)\pmod4$ ($1\le i<j\le n$) and 
the following system has a solution over $\mathbb{Z}$. 
\[
\left\{
\begin{array}{l}
\displaystyle4\left(\mu_{\sigma'}(jiij)-\mu_{\sigma}(jiij)\right)+\mu_{\sigma'}(ij)-\mu_{\sigma}(ij)
\equiv 4\mu_{\sigma}(ij)(x_{ij}+\!x_{ji}) \pmod8 \\
 \qquad \qquad\qquad\qquad\qquad\qquad\qquad\qquad\qquad\qquad\qquad\qquad\qquad\qquad  (1\leq i<j \leq n)\\
\mu_{\sigma'}(ijk)-\mu_{\sigma}(ijk) 
\equiv \mu_{\sigma}(ij)(x_{ki}-x_{kj})+\mu_{\sigma}(jk)(x_{ij}-x_{ik})
+\mu_{\sigma}(ki)(x_{jk}-x_{ji})  \pmod2 \\
 \qquad \qquad\qquad\qquad\qquad\qquad\qquad\qquad\qquad\qquad\qquad\qquad\qquad\qquad (1\leq i<j<k\leq n)
\end{array}
\right.
\]
\end{itemize}
\end{theoB*}
As a direct corollary of Theorem~2, we recover the band-pass  classification of algebraically split links  given in \cite[Thm.~4.1]{Martin}: the classification is given by the invariants $a_2(i)\pmod2$, $\mu(jiij)\pmod2$ and $\mu(ijk)$ 
($ i<j<k$).
\\
We also recover the band-$\shp$ classification of $2$-component {\em $\mathbb{Z}_2$-algebraically split links}, \emph{i.e.} links with
pairwise even linking numbers, due to Saito \cite{Saito}. In fact, we have the following generalization of \cite[Thm.~3.2]{Saito} to any number of components.
For all $i,j~(i<j)$, set 
$$\varphi_{\sigma}(ij)\equiv 4\mu_{\sigma}(jiij)+\mu_{\sigma}(ij) \mod 8.$$
\begin{cor}
Let $L$ and $L'$ be $n$-component $\mathbb{Z}_2$-algebraically split links, and let $\sigma$ and $\sigma'$ be 
bottom tangles whose closures $\widehat{\sigma}$ and $\widehat{\sigma'}$ 
are equivalent to $L$ and $L'$ respectively. 
Then $L$ and $L'$ are band-$\shp$ equivalent if and only if $\sigma$ and $\sigma'$ share all invariants $\mu(ij)\pmod4$, 
$\varphi(ij)\pmod8$ and $\mu(ijk)\pmod2$ 
 $(i<j<k)$.
\end{cor}
\medskip

Our strategy for the proofs of the classification Theorems~1 and 2 
 is roughly inspired from Habegger and Lin's  approach to the link-homotopy classification of links. 
First, we give classification results for bottom tangles (or equivalently, string links) 
up to clasp-pass, band-pass moves and band-$\shp$ moves (Section \ref{sec:classtangles}),  
by providing a complete set of invariants for each equivalence relation. 
Second, we give  in Section \ref{sec:Markov} a \lq Markov type' theorem (Theorem~\ref{Markov}) for each of these relations. 
Finally, we combine these results, by analyzing the behavior of the classifying invariants of bottom tangles, under the \lq Markov type' moves (Proposition~\ref{trace}).

\begin{rem} 
The following diagram summarizes the various known implications among all equivalence relations on knot and links, discussed in this introduction.   
\[\hspace{-1.5ex}\begin{array}{ccccccccc}
\begin{array}{c}\textrm{\small ambient} \\ \textrm{\small isotopy}\end{array}\hspace{-0.75ex}&\hspace{-0.75ex}\Longrightarrow\hspace{-0.75ex}&\hspace{-0.75ex} \text{\small concordance}\hspace{-0.75ex}&\hspace{-0.75ex}\stackrel{\strut \textrm{\cite{Shibuya}}}{\Longrightarrow} \hspace{-0.75ex}&\hspace{-0.75ex}\begin{array}{c}\textrm{\small \,\,\,self-pass\,}^{\textrm{\cite{S-Y}}}\\ \textrm{\small{equivalence}}\end{array}\hspace{-0.75ex}&\hspace{-0.75ex}\stackrel{\strut \textrm{\cite{M-N}}}{\Longrightarrow} &\begin{array}{c}\textrm{\small \,\,self-$\shp$\,}^{\textrm{\cite{S-Y}}}\\ \textrm{\small{equivalence}}\end{array}\hspace{-0.75ex}&\hspace{-0.75ex}\Longrightarrow \hspace{-0.95ex}&\hspace{-0.95ex}\begin{array}{c}\textrm{\small \,\,link-\,}^{\textrm{\cite{HL}}} \\\textrm{\small{homotopy}}\end{array}\\
\Downarrow \hspace{-0.75ex}&\hspace{-0.75ex}\Searrow\hspace{-0.75ex}&\hspace{-0.75ex}& &\Downarrow& & \Downarrow & & \\
\begin{array}{c}\textrm{\small delta\,}^{\textrm{\cite{M-N}}} \\\textrm{\small{equivalence}}\end{array} \hspace{-0.75ex}&\hspace{-0.75ex}\stackrel{\strut \textrm{\cite{Habiro}}}{\Longleftarrow}\hspace{-0.75ex}&\hspace{-0.75ex} \begin{array}{c}\textrm{\small clasp-pass\,}^{\textrm{Th.1}} \\\textrm{\small{equivalence}}\end{array}\hspace{-0.75ex}&\hspace{-0.75ex}\Longrightarrow \hspace{-0.75ex}&\hspace{-0.75ex}\begin{array}{c}\textrm{\small band-pass\,}^{\textrm{Th.2}} \\\textrm{\small{equivalence}}\end{array}\hspace{-0.75ex}&\hspace{-0.75ex}\stackrel{\strut \textrm{\cite{M-N}}}{\Longrightarrow} \hspace{-0.75ex}&\hspace{-0.75ex}\begin{array}{c}\textrm{\small band-$\shp$\,}^{\textrm{Th.2}} \\\textrm{\small{equivalence}}\end{array}\hspace{-0.75ex}&  & 
\end{array}\]
\noindent There, $X\Rightarrow Y$ means that if two links are equivalent up to $X$, then they also are equivalent up to $Y$. The known classifications results are also shown as exponents of the corresponding equivalence relations. 
\end{rem}

\bigskip

\begin{convention}\label{conv}
Given two subsets $I,J$ of $\{1,\ldots, n\}$, we denote by  $\delta_{I,J}$ the Kronecker symbol, defined by $\delta_{I,J}=1$ if  $I=J$ and $\delta_{I,J}=0$ otherwise. If $I=\{i\}$ and $J=\{j\}$ for some indices $i,j$, we simply denote $\delta_{I,J}=\delta_{i,j}$. 
\end{convention}

\begin{acknowledgments}
The authors thank Emmanuel Graff and Kodai Wada for useful comments on a draft version of this paper. 
The authors are also indebted to the referee for his/her careful reading of the manuscript and insightful comments. 
\end{acknowledgments}

\section{Clasper calculus up to clasp-pass moves}

\subsection{Claspers and $C_k$-moves}

Let $\sigma$ be a bottom tangle in $D^2\times [0,1]$. 
We first recall from \cite{Habiro} the main definition of this section; we stress that Habiro gave a more general definition of claspers in \cite{Habiro}, but we only need the following restricted notion in this paper. \begin{defi} 
A {\em tree clasper} for $\sigma$ is a disk $T$ embedded in $D^2\times [0,1]$ and satisfying the following three conditions:
\begin{enumerate}
\item[(i)] $T$ decomposes into disks and bands, called {\em edges}, each of which 
connects two distinct disks;
\item[(ii)] the disks have either 1 or 3 incident edges, and are called {\em leaves} or {\em nodes} respectively;
\item[(iii)] $T$ intersects $\sigma$ transversely, and the intersection is contained in the union of 
the interior of the leaves. 
\end{enumerate}
A leaf of $T$ is {\em simple} if it intersects $\sigma$ at a single point, and the tree clasper $T$ is \emph{simple} if all of its leaves are. 
The degree of $T$ is defined as the number of leaves minus 1. A degree $k$ tree clasper is called a \emph{$C_k$-tree}. 
\end{defi} 
In this paper, we follow the drawing convention for claspers in \cite[Figure 7]{Habiro}, except for the following: a $\oplus$ (resp. $\ominus$) on some edge denotes a positive (resp. negative) half-twist.\footnote{This graphical notation replaces the circled $s$ (resp. $s^{-1}$) used in \cite{Habiro}.}
\medskip

Given a disjoint union $F$ of  tree claspers  for $\sigma$, there is a procedure to construct a 
framed link $L(F)$ in the complement of $\sigma$.  
Since the result $(D^2\times [0,1])_{L(F)}$ of surgery along $L(F)$ 
is diffeomorphic to  $D^2\times[0,1]$, we have a new tangle $\sigma_F$ such that 
$(D^2\times[0,1],\sigma_F)\cong((D^2\times[0,1])_{L(F)},\sigma)$, relative to the boundary.
\begin{defi}
 We say that $\sigma_F$ is obtained from $\sigma$ by {\em surgery along $F$}.
 In particular, surgery along a simple $C_k$-tree defines a local move on tangles as illustrated in Figure~\ref{fig:C5}, which is called {\em $C_k$-move}.
 \begin{figure}[!h]
\includegraphics[scale=0.85]{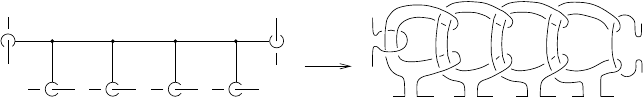} 
  \caption{Surgery along a simple $C_5$-tree}\label{fig:C5}
\end{figure}

\noindent The $C_k$-equivalence, denoted by $\stackrel{C_k}{\sim}$, is the equivalence relation on tangles generated by surgeries along $C_k$-trees and isotopies. 
\end{defi}

Habiro showed that the $C_k$-equivalence relation becomes finer as $k$ increases \cite{Habiro}. 

\begin{exemple}\label{Ex:Cn}
A $C_1$-move corresponds to a crossing change. Hence any bottom tangle is obtained from the trivial one by surgery along a union of $C_1$-trees. Moreover, the band-pass and band-$\shp$ moves of the introduction can be reformulated in the following way: 
 $$ \textrm{\includegraphics[scale=0.85]{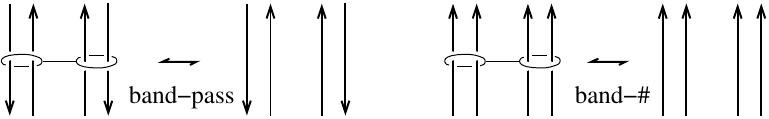}}$$ 

Figure \ref{fig:exclasp} gives several  examples of surgeries along simple $C_2$-trees.
As the right-hand side of the figure shows, a $C_2$-move corresponds to a band sum with a copy of the Borromean rings, hence is equivalent to the {\em $\Delta$-move} defined by Matveev \cite{Matveev} and Murakami and Nakanishi~\cite{M-N} independently.
 \begin{figure}[!h]
\includegraphics[scale=0.85]{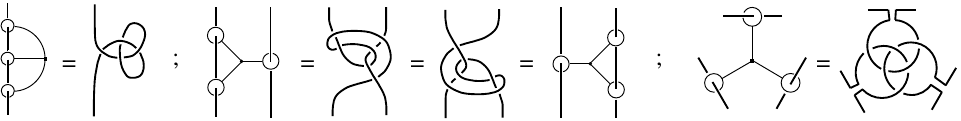} 
  \caption{Examples of surgeries along simple $C_2$-trees}\label{fig:exclasp}
\end{figure}

A $C_3$-move is equivalent to the clasp-pass move defined in the introduction: a detailed proof of this equivalence can be found in Figure 3.2 of \cite{Tsukamoto}. 
We shall make implicitly use of this latter equivalence throughout the paper.
\end{exemple}

We shall also need the following notion from \cite{MY}.
\begin{defi}
 The {\em $C_k$-concordance} ($k\ge 1$) is the equivalence relation on tangles, denoted by $\stackrel{C_k+c}{\sim}$, generated by $C_k$-equivalence and concordance.
\end{defi}

This equivalence relation turns out to coincide with the order $k$ 
{\em Whitney concordance} studied by Conant, Schneiderman and Teichner in \cite{CST}; a proof that these two notions indeed coincide can be found in \cite{CST2}.

\subsection{Calculus of claspers up to $C_3$-equivalence}

Habiro gives in \cite[Prop.~2.7]{Habiro} a list of twelve moves on claspers which yield \emph{equivalent} claspers, that is, claspers with homeomorphic surgery effect.  In what follows, we will freely use these \emph{Habiro moves} by referring to their numbering in \cite[Fig.~8~and~9]{Habiro}, and we will simply denote two equivalent claspers by \lq =\rq.

Habiro further develops this \emph{calculus of claspers} in \cite[\S 4]{Habiro}, through a collection of results analyzing elementary operations on claspers up to  $C_k$-equivalence. 
In the next result, we summarize Habiro's results, and several easy consequences, for the case of interest in this paper, namely up to $C_3$-equivalence. 

\begin{lem}[Calculus of claspers up to $C_3$-equivalence]\label{lem:calculus} 
$\textrm{ }$Let $F$ be a disjoint union of tree claspers for a bottom tangle $\sigma$. 
\begin{enumerate}
 \item[(1)]
 If $F'$ is obtained from $F$ by passing an edge of some tree clasper across an edge of another tree clasper in $F$, then $\sigma_{F}\stackrel{C_3}{\sim} \sigma_{F'}$. 
 \item[(2)]
 Suppose that $F'$ is obtained from $F$ by passing an edge of some $C_k$-tree in $F~(k\ge 1)$ 
 across a strand of $\sigma$.     
     $$\qquad \textrm{\includegraphics[scale=0.8]{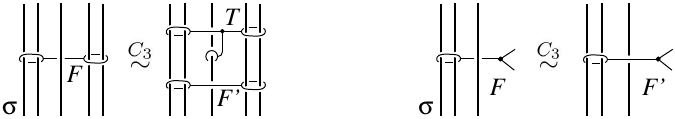}}$$ 
     \begin{itemize}
     \item if $k=1$, then $\sigma_{F}\stackrel{C_3}{\sim} \sigma_{F'\cup T}$, where $T$ is a $C_2$-tree as shown on the left-hand side of the figure above;\footnote{In this figure, parallel bold strands intersecting leaves represent any number of parallel strands (possibly a single one) of the tangle $\sigma$; we make use of the same convention in all figures of Lemma~\ref{lem:calculus}.}  
     \item if $k\ge 2$, then $\sigma_{F}\stackrel{C_3}{\sim} \sigma_{F'}$ (see right-hand side).  
     \end{itemize}
 \item[(3)]  Suppose that $F'$ is obtained from $F$ by passing a leaf of some $C_k$-tree across a leaf of another $C_{k'}$-tree ($k,k'\ge 1$). %, as illustrated below.     
     $$\qquad \textrm{\includegraphics[scale=0.8]{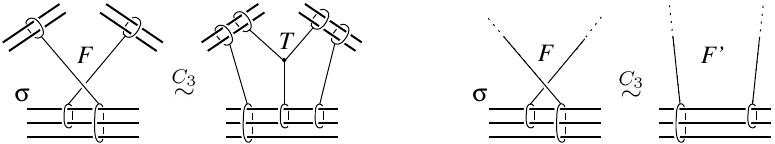}}$$ 
     \begin{itemize}
     \item if $k=k'=1$ $($left-hand side$)$, then $\sigma_{F}\stackrel{C_3}{\sim} \sigma_{F'\cup T}$, where $T$ is a $C_2$-tree as shown in the figure above;
     \item otherwise $($right-hand side$)$, we have $\sigma_{F}\stackrel{C_3}{\sim} \sigma_{F'}$.  
     \end{itemize}
 \item[(4)]      Let $f_1$ and $f_2$ be two disks obtained by splitting a leaf $f$ of some $C_k$-tree $T$ in $F$ $(k=1$ or $2)$, and let $T_1$ and $T_2$ be obtained from $T$ by replacing leaf $f$ with $f_1$ and $f_2$, respectively, see figure below. 
 Suppose that $F'$ is obtained from $F$ by replacing $T$ with $T_1\cup T_2$. 
     $$\qquad \textrm{\includegraphics[scale=0.8]{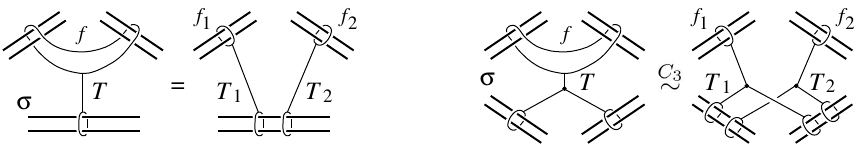}}$$ 
     \begin{itemize}
     \item if $k=1$ $($left-hand side$)$, then $\sigma_{F}=\sigma_{F'}$; 
     \item if $k=2$ $($right-hand side$)$, then $\sigma_{F}\stackrel{C_3}{\sim} \sigma_{F'}$.  
     \end{itemize}
 \item[(5)]  Suppose that $F'$ is obtained from $F$ by deleting a pair of $C_k$-trees $(k=1$ or $2)$ that only differ by a positive half-twist as shown below. 
     $$\qquad \textrm{\includegraphics[scale=0.8]{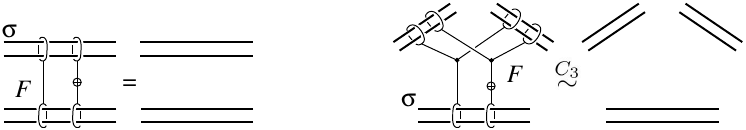}}$$ 
     \begin{itemize}
     \item if $k=1$ $($left-hand side$)$, then $\sigma_{F}=\sigma_{F'}$;
    
     \item if $k=2$ $($right-hand side$)$, then $\sigma_{F}\stackrel{C_3}{\sim} \sigma_{F'}$.  
     \end{itemize}
 \item[(6)] Suppose that $F$ and $F'$ only differ in a neighborhood of a node, in any of the ways represented below.  
  Then $\sigma_{F}\stackrel{C_3}{\sim} \sigma_{F'}$. 
           $$\qquad \textrm{\includegraphics[scale=1]{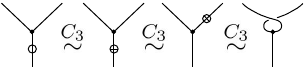}}$$ 
\end{enumerate}
\end{lem}

\begin{proof}[Sketch of proof]
The proofs are omitted, as they use the same techniques as in \cite[\S 4]{Habiro};  
see also \cite{Goussarov_graphs,GGP} and \cite[App.~E]{Ohtsuki_book}, where similar statements appear. 
Precise references are as follows.
(1) is given in \cite[Prop.~4.6]{Habiro};
(2) is proved in \cite[Prop.~4.5]{Habiro};
(3) follows from the proof \cite[Prop.~4.4]{Habiro};
the case $k=1$ of (4) follows from Habiro's moves 6 and 7, and the case $k=2$ is implicit in \cite[pp.~26]{Habiro};
(5) is a consequence of Habiro's move 4, see \cite[Lem.~E.7]{Ohtsuki_book};
The first two equivalences of (6) follow from the case $k=2$ of (5), and the last equivalence  follows from the former two as shown in \cite[Lem.~E.9]{Ohtsuki_book}. 
\end{proof}

We conclude by noting that clasper calculus up to concordance is in practice more flexible.
This is in particular illustrated by the following result, see \cite[Lem.~4.2]{MYabelian}.
\begin{lem}\label{lem:conc}
 Let $C$ be a single $C_k$-tree for the tangle $\sigma$ $(k\ge 1)$, which contains two simple leaves intersecting the same  component as illustrated below. 
 Let $C'$ be obtained by exchanging the relative positions of these two leaves, as shown in the figure. 
 Then $\sigma_{C'}\stackrel{C_{k+1}+c}{\sim} \sigma_C$.
  $$ \textrm{\includegraphics[scale=1.75]{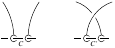} }$$
\end{lem}

\subsection{Doubled and parallel $C_k$-trees}

We now introduce two specific types of claspers, that will play central roles in the rest of this paper.

\begin{defi}\label{def:doubled}
 A tree clasper for the tangle $\sigma$ is a \emph{doubled tree clasper} if all its leaves are simple except for one, called \emph{doubled leaf}, that intersects $\sigma$ at exactly two points of the same component. 
 A doubled tree clasper is \emph{of parallel type} if its doubled leaf  intersects $\sigma$ at two points with same sign; otherwise, we say that it is \emph{of antiparallel type}. 
 \end{defi}
\begin{defi}\label{def:parallel}
A \emph{pair of parallel tree claspers} for the tangle $\sigma$ is a union of two parallel copies of some tree clasper.\footnote{We note that there is no ambiguity here in the notion of parallel copies, since for a tree clasper the underlying surface is a disk.}
\end{defi}

The next two technical results focus on doubled $C_2$-trees. There are two cases, depending on the type of the doubled tree.

\begin{lem}\label{basicmove2bis}
Let $D$ be a doubled $C_2$-tree of parallel type for a tangle $\sigma$. Then $\sigma_D$ is $C_3$-concordant to $\sigma_P$, 
where $P$ is a pair of parallel $C_2$-trees as shown below. 
$$ \textrm{\includegraphics[scale=1]{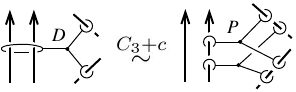}} $$
\end{lem}
\begin{proof}
 Without loss of generality, we can assume that we are in the situation depicted on the left-hand side of Figure~\ref{fig:parallel}.
 \begin{figure}[!h]
\includegraphics[scale=0.9]{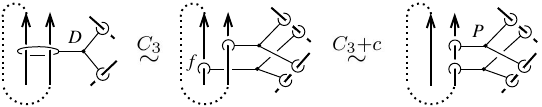} 
  \caption{Proof of Lemma~\ref{basicmove2bis}}\label{fig:parallel}
\end{figure}
 Applying Lemma \ref{lem:calculus}~(4) to split the doubled leaf, we  may replace $D$ with two simple $C_2$-trees as shown in the center of the figure, up to $C_3$-equivalence. Next, we can slide the simple leaf denoted by $f$ in the figure, following the orientation of $\sigma$ until we obtain the desired pair of parallel $C_2$-trees. 
 This sliding can be achieved freely up to $C_3$-concordance, thanks to Lemma \ref{lem:calculus}~(1)-(3) and Lemma~\ref{lem:conc}. This latter lemma is necessary in the case where another leaf of the $C_2$-tree containing $f$, is met along the arc along which we are sliding $f$. 
\end{proof}

\begin{lem}\label{basicmove2}
Let $D$ be a doubled $C_2$-tree of antiparallel type for a tangle $\sigma$. 
The two intersection points of $\sigma$ with the doubled leaf then cobound an arc $\alpha$ in $\sigma$. 
Denote also by $f$ and $g$ the other two leaves of $D$, see the left-hand side of Figure~\ref{double-arc}.
We have the following:
\begin{enumerate}
\item[(i)] If $\alpha\cap f=\alpha\cap g=\emptyset$, then $\sigma_D$ is $C_3$-equivalent to $\sigma$.
\item[(ii)] If $\alpha\cap f\neq \emptyset$ and $\alpha\cap g\neq \emptyset$, then $\sigma_D$ is $C_3$-equivalent to $\sigma$.
\item[(iii)] Otherwise, \emph{i.e.} if only one of $f$ and $g$ intersects $\alpha$, then $\sigma_D$ is $C_3$-concordant to $\sigma$.
 \end{enumerate}
 In particular, we always have that $\sigma_D$ is $C_3$-concordant to $\sigma$. 
\end{lem}
\begin{proof}
Figure~\ref{double-arc} summarizes the proof of the three cases. 
\begin{figure}[!h]
\includegraphics[scale=0.9]{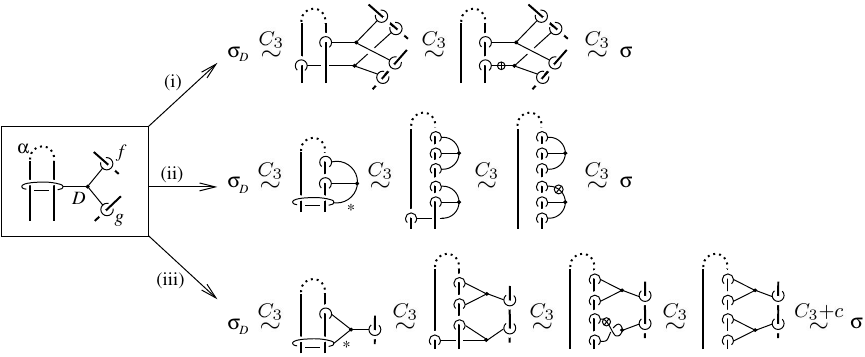} 
  \caption{An $\ast$-label near some edge, 
  expresses the fact that the equivalence yields either the depicted clasper, or the one obtained by inserting a positive half-twist on the $\ast$-marked edge.}\label{double-arc}
\end{figure}

Case (i) is illustrated in the top line.
There, the first equivalence is given by Lemma \ref{lem:calculus}~(4), while the second equivalence follows from Lemma \ref{lem:calculus}~(1)-(3). The result then follows from Lemma \ref{lem:calculus}~(5).

The second line gives the proof of case (ii).
By Lemma \ref{lem:calculus}~(1)-(3) and (6), we obtain the first equivalence.\footnote{As the caption of Figure~\ref{double-arc} indicates, we actually obtain either the depicted clasper, or the one obtained by inserting a positive half-twist on the $\ast$-marked edge: the rest of the figure only gives the proof in the untwisted case, the other case being strictly similar. A similar convention is used in the proof of case (iii).}
The rest of the argument is then exactly the same as in case (i). 

Finally, we prove case (iii), following the third line in Figure~\ref{double-arc}.
We first apply Lemma \ref{lem:calculus}~(1)-(3), then (4) and (1)-(3) as in case (ii), to get the first three equivalences. Applying Lemma~\ref{lem:calculus}~(6) then gives the fourth equivalence. The final step is then given by combining Lemmas~5.1~(1) and 5.2 of \cite{MY}.
\end{proof}

For later use we note the following, which was used in the proof of case (iii). 
\begin{cor}\label{cor}
We have the following equivalence.\\[-0.3cm]
$$ \textrm{\includegraphics[scale=1]{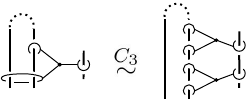}}$$
\end{cor}

\subsection{The clasper index}

We conclude this section with an extra notion associated with claspers, that will be used throughout the rest of this paper.

Let $\sigma$ be an $n$-component bottom tangle.
Let $T$ be a simple $C_1$-tree for $\sigma$. 
If $T$ intersects the  $i$th and $j$th components of $\sigma$  
as illustrated on the left-hand side of Figure~\ref{index} ($1\le i,j\le n$, possibly $i=j$), 
we say that $T$ has {\em index $(i,j)$}.  
\begin{figure}[!h]
\includegraphics[scale=1]{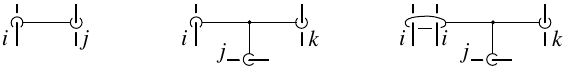}
\caption{Tree claspers of index $(i,j)$ (left), $(i,j,k)$ (center) and $(i^{(2)},j,k)$ (right)}\label{index}
\end{figure}
Likewise, let $T$ be a simple $C_2$-tree for $\sigma$. 
If $T$ intersects the $i$th, $j$th and $k$th components of $\sigma$ 
as illustrated in the center of Figure~\ref{index} ($i,j,k\in\{1,...,n\}$), 
we say that $T$ has {\em index $(i,j,k)$}.  Note that the index of a $C_2$-tree is defined only up to cyclic permutations.

This notion extends to doubled $C_2$-trees, as follows. If the doubled leaf of some doubled tree $T$ intersects twice the $i$th component of $\sigma$, then we write $i^{(2)}$ instead of $i$ when writing the index of $T$. For example, the doubled $C_2$-tree represented on the the right-hand side of Figure~\ref{index} has index $(i^{(2)},j,k)$.

\section{Invariants}

This section reviews the definition and relevant properties of the invariants involved in our classification results.

\subsection{Invariants for bottom tangles}

\subsubsection{The Casson knot invariant $a_2$}

The \emph{Alexander-Conway polynomial} of oriented links is a polynomial  $\nabla$ in the variable $z$ defined by the fact that $\nabla_U(z)=1$, where $U$ denotes the unknot, and the skein relation 
 $$\nabla_{\Ov}(z) - \nabla_{\und}(z) = z\nabla_{\smoo}(z), $$
where $\Ov$, $\und$ and $\smoo$ are three links that are identical except in a $3$-ball where they look as depicted.

The Alexander-Conway polynomial of a knot $K$ has the form 
$$\nabla_K(z)= 1 + \sum_{k\ge 1} a_{2k}(K) z^{2k}. $$
In particular, the coefficient $a_2$ of $z^2$ in $\nabla$ is a knot invariant, called the \emph{Casson knot invariant}. 
\medskip

Now, we can use the closure-type operations of Figure~\ref{closure} to define several invariants of bottom tangles, as follows. 
\begin{defi} 
 Let $\sigma$ be an $n$ component bottom tangle, and let  $i,j\in \{1,\ldots,n\}$ such that $i\neq j$.
 \begin{itemize}
  \item $a_2(\sigma;i)$ denotes the Casson knot invariant $a_2$ of the closure $\widehat{\sigma_i}$.  
  \item $a_2(\sigma;ij)$ denotes the Casson knot invariant $a_2$ of the plat closure $p_{ij}(\sigma)$. 
 \end{itemize}
\end{defi}
\noindent 
As in the introduction, we simply denote these invariants by $a_2(i)$ and $a_2(ij)$.  
Note that the latter invariant $a_2(ij)$ was first introduced in \cite{Meilhan}, where it  appears under the notation $V_2$. 

\subsubsection{Milnor invariants} \label{sec:milnor}

We now briefly review Milnor invariants of bottom tangles, following \cite{Levine}.\footnote{Note that Levine use in \cite{Levine} the terminology \lq string links\rq~ for what we call here bottom tangles.} 
We  also refer the reader to \cite{HL} for the equivalent case of string links, and to \cite{M1,M2} for the link case originally treated by Milnor.

Let $\sigma=\sigma_1\cup \cdots \cup \sigma_n$ be an $n$-component bottom tangle in $D^2\times [0,1]$. 
Fix a basepoint $\ast$ in $\partial D^2\times \{0\}$, and denote by $G(\sigma)$ the fundamental group of the complement of $\sigma$ in $D^2\times [0,1]$. 
Denote by $m_1, \ldots, m_n$ the canonical \emph{meridian} elements of $G(\sigma)$ represented by the $n$ arcs sitting in $D^2\times \{0\}$ as shown on the left-hand side of Figure~\ref{fig:bottom}.
\begin{figure}[!h]
\includegraphics[width=8cm]{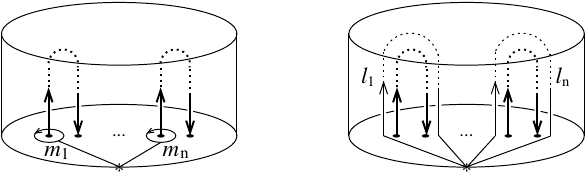}
  \caption{Meridians and longitudes of a bottom tangle}\label{fig:bottom}
\end{figure}
\noindent Denote also by $l_1,\ldots,l_n\in G(\sigma)$ the longitude elements, which are represented by parallel copies of each component of $\sigma$ as shown on the right-hand side of Figure~\ref{fig:bottom}; here we consider \emph{preferred longitudes}, which are determined by the fact that each $l_i$ is trivial in the abelianization $G(\sigma)/[G(\sigma),G(\sigma)]$.
\medskip

As usual in knot theory, a presentation for $G(\sigma)$ can be extracted from a diagram of $\sigma$, where each arc of the diagram yields a \emph{Wirtinger generator}, and each crossing yields a conjugation relation. But the situation turns out to be much simpler when quotienting by the lower central series subgroups.

The \emph{lower central series} of a group $G$ is the descending series $\left(G_q\right)_{q\ge 1}$ of subgroups, defined inductively by $G_1=G$ and $G_{q+1}=[G,G_q]$. 
Let $F$ denote the free group on $n$ generators $x_1,\cdots, x_n$. 
By a theorem of Stallings \cite[Thm.~5.1]{stall}, the map $F\longrightarrow G(\sigma)$ defined by sending $x_i$ to the meridian element $m_i$ for each $i$, induces isomorphisms 
$$ \frac{F}{F_q}\xrightarrow{\simeq} \frac{G(\sigma)}{G(\sigma)_q} $$ 
for all $q\ge 2$. 

For a fixed level $q\ge 2$, denote by $l^q_i\in F/F_q$ the image of the $i$th preferred longitude through this isomorphism. 
The \emph{Magnus expansion} of $l^q_i$ is the
formal power series $E(l^q_i)$ in non-commuting variables $X_{1},\ldots,X_{n}$, obtained by substituting $1+X_j$ for $x_j$, and $1-X_j+X_j^2-X_j^3+\ldots$ for $x_j^{-1}$.
\begin{defi}
For any sequence $I=i_{1}\ldots i_{k-1}i$ of indices in $\{1,\ldots,n\}$ ($k\leq q$), the {\it Milnor invariant} $\mu_{\sigma}(I)$ is the coefficient of $X_{i_{1}}\cdots X_{i_{k-1}}$ in $E(l^q_i)$. 
The \emph{length} of this Milnor invariant is the length $k$ of the sequence $I$.
\end{defi}
Length $1$ Milnor invariants are set to be zero by convention. 
Length $2$ Milnor invariants $\mu_\sigma(ij)$ ($1\le i\neq j\le n$) coincide with the linking number of components $i$ and $j$, and $\mu(ijk)$ is therefore sometimes referred to as the triple linking number.
\begin{rem}\label{rem:wellwellwell}
Let $\sigma$ be a bottom tangle with vanishing linking numbers. 
Then it is well known that $\mu_{\sigma}(I)=0$ for length $\le 3$ sequences $I$ involving only $1$ or $2$ distinct indices. Indeed if one such invariants $\mu(I)$ were nontrivial, then as first non-vanishing Milnor invariant of $\sigma$ it would be equal to the invariant $\overline{\mu}(I)$ of the closure $\widehat{\sigma}$; but the latter is necessarily trivial \cite[\S~4]{M2}, leading to a contradiction. 
\end{rem}

\subsection{Properties}

In this subsection, we gather the main properties of the invariants defined above, that will be used in the proofs of our main results. 
Let us begin by recalling, in the next two lemmas, some known topological invariance properties.

Recall that \emph{link-homotopy} is an equivalence relation on tangles, generated by isotopies and self-crossing changes, that is, crossing changes involving two strands of a same component \cite{M2}.  
The following combines results of Stallings \cite{stall}, Casson \cite{Casson}, and Habegger-Lin \cite{HL}.
\begin{lem}\label{lem:lh}
 For any sequence $I$ of indices in $\{1,\ldots,n\}$, Milnor invariant $\mu(I)$ is a concordance invariant.
 Moreover, if each index in $\{1,\ldots,n\}$ appears at most once in $I$, then 
$\mu(I)$ is a link-homotopy invariant.
\end{lem}

\begin{lem}\label{lem:C3inv}
 For any $i,j,k~(i< j< k)$, 
 $a_2(i)$, $a_2(ij)$, $\mu(ij)$, $\mu(jiij)\pmod2$ and $\mu(ijk)$ are all invariants of $C_3$-equivalence.
\end{lem}
\begin{proof}
  It is shown in \cite[Thm.~7.2]{Habiro} that all Milnor invariants of length $\le 3$ are invariants of $C_3$-equivalence. 
  The fact that the Casson knot invariant $a_2$ is a $C_3$-equivalence invariant, is also shown in \cite{Habiro}; this is a consequence of a more general result stating that finite type invariants of degree $k$ are $C_{k+1}$-equivalence invariants \cite[\S~6]{Habiro}. 
  It readily follows that the invariants $a_2(i)$ and $a_2(ij)$ also  are $C_3$-equivalence invariants.
  Finally, the result for $\mu(jiij)\pmod2$ is given in \cite[Rem.~5.4]{MY}. 
\end{proof}

 Using  Lemma~\ref{lem:C3inv}, we next show the following variation formulas. 
\begin{lem}\label{keylemma}
Let $\sigma$ be an $n$-component bottom tangle. 
Let $i,j,k$ be indices such that $1\leq i\le j\le k \leq n$.
Let $C_{ijk}$, resp. $C_{ijk}^{-1}$ be a simple $C_2$-tree of index $(i,j,k)$ for $\sigma$, such that there exists a $3$-ball in $D^2\times [0,1]$ intersecting $\sigma\cup C_{ijk}$, resp. $\sigma\cup C_{ijk}^{-1}$, as shown on the left-hand side, resp. right-hand side, of the figure below.
$$ \textrm{\includegraphics[scale=1]{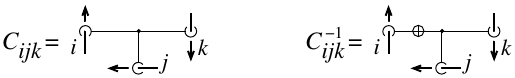}} $$
We have the following variation formulas, for $\varepsilon=\pm 1$: 
\begin{enumerate}
\item[(i)]  $a_2(\sigma_{C^{\varepsilon}_{iii}};r)-a_2(\sigma;r)=\varepsilon \delta_{i,r}$, for all $r\in \{1,\ldots,n\}$;
\item[(ii)]  If $i<j$, $a_2(\sigma_{C^{\varepsilon}_{iij}};rs)-a_2(\sigma;rs)=\varepsilon \delta_{\{i,j\},\{r,s\}}$, for all $r,s$ such that  $1\le r<s\le n$;
\item[(iii)]  $\mu_{\sigma_{C^{\varepsilon}_{iij}}}(srrs)-\mu_{\sigma}(srrs)\equiv \delta_{\{i,j\},\{r,s\}}\pmod2$, for all $r,s$ such that  $1\le r<s\le n$;
\item[(iv)] $\mu_{\sigma_{C^{\varepsilon}_{ijk}}}(rst)-\mu_{\sigma}(rst)=- \varepsilon\delta_{\{i,j,k\},\{r,s,t\}}$, for all $r,s,t$ such that  $1\le r<s<t\le n$.
\end{enumerate}
\end{lem} 
\begin{proof}
We first observe that it suffices to prove these formulas for the case $\varepsilon=1$; the case $\varepsilon=-1$ follows, using Lemma~\ref{lem:calculus}~(5) and the $C_3$-invariance Lemma \ref{lem:C3inv}.

Let us first prove (i). It is clear that $a_2(\sigma_{C_{iii}};r)=a_2(\sigma;r)$ if $r\neq i$, since $C_{iii}$ becomes trivial when removing the $i$th component $\sigma_i$ of $\sigma$, by Habiro's move 1. Note that $\sigma$ can be regarded as obtained from the trivial bottom tangle by surgery along a union of simple $C_1$-trees. Hence by Lemma~\ref{lem:calculus}, the $i$th component of $\widehat{\sigma_{C_{iii}}}$ is $C_3$-equivalent to the connected sum $\widehat{\sigma_i}\shp \hat{T}$, where $T$ is shown on the left-hand side of Figure~\ref{fig:exclasp}. 
Since the Casson knot invariant of the trefoil knot is one, and using the fact that $a_2$ is additive under connected sum and is a $C_3$-equivalence invariant (Lemma \ref{lem:C3inv}), the result follows. Formula (ii) follows from the very same arguments, since the image of $C_{iij}$ under the plat-cloture $p_{ij}$, is a copy of $C_{iii}$.\\
We now turn to (iii).
As above, by Habiro's move 1, the case $\{i,j\}\neq \{r,s\}$ is clear, so we only consider the case $\{i,j\}=\{r,s\}$. 
For convenience, we give here the proof for the equivalent situation where $\sigma$ is an $n$-component string link, which we may regard as obtained from the trivial string link $\1$ by surgery along a union of tree claspers. 
Now, using Lemma~\ref{lem:calculus}~(1)-(3), we may 'isolate' $C_{iij}$ up to $C_3$-equivalence, in the sense that we have $\sigma_{C_{iij}}\stackrel{C_3}{\sim} \sigma\cdot \1_{C_{iij}}$. Hence using the $C_3$-equivalence invariance Lemma \ref{lem:C3inv}, it  suffices to show that $\mu_{\sigma\cdot \1_{C_{iij}}}(jiij)\equiv \mu_{\sigma}(jiij)+ 1\pmod2$. 
 Observe that $\1_{C_{iij}}$ is link-homotopic to $\1$.\footnote{This follows from the main result of \cite{FY}, but follows also easily from Figure~\ref{fig:exclasp}.} Hence by Lemma~\ref{lem:lh} and Remark~\ref{rem:wellwellwell} we have that $\mu_{\1_{C_{iij}}}(I)=0$ for any length $3$ sequence $I$. 
 Using the additivity property of Milnor invariants \cite[Lem.~3.3]{MYpacific}, this implies that $\mu_{\sigma_{C_{iij}}}(J)=\mu_{\sigma}(J)+\mu_{\1_{C_{iij}}}(J)$ for any length $4$ sequence $J$. The result then directly follows from \cite[Lem.~5.3~(3)]{MY}. \\
 Formula (iv) for Milnor's triple linking number $\mu(rst)$ relies on a similar (though somewhat  simpler) argument, using the fact that surgery along $C_{ijk}$ is a band sum with a copy of the Borromean rings (Figure~\ref{fig:exclasp}) whose triple linking number $\mu(ijk)$ equals $-1$.
\end{proof}
\begin{rem}\label{remSL}
 Building on the proof of (iii), we can further deduce that, if $C$ is an index $(i,i,j)$ $C_2$-tree for the trivial string link $\1$, then $\mu_{\1_C}(jiij)\equiv  \mu_{\1_C}(iijj)\equiv  \beta\left(\widehat{\1_C}\right) \pmod2$, where $\beta$ denotes the \emph{Sato-Levine invariant}. Indeed, the $i$th and $j$th components of $\widehat{\1_C}$ form a copy of the Whitehead link (Figure~\ref{fig:exclasp}), which has Sato-Levine invariant $1$. 
\end{rem}

The following lemma follows directly from Lemma~\ref{keylemma}.
\begin{lem}\label{lem:milnorparallel}
 Let $P$ be a pair of parallel $C_2$-trees for an $n$-component bottom tangle $\sigma$. 
 Then for any $i,j,k~(i<j<k)$, we have 
 $a_2(\sigma_P;i)\equiv a_2(\sigma;i) \pmod2$, $\mu_{\sigma_P}(ijk)\equiv \mu_{\sigma}(ijk) \pmod2$ and $\mu_{\sigma_P}(jiij)\equiv \mu_{\sigma}(jiij) \pmod2$.
\end{lem}

In \cite[\S 3]{M-W-Y}, Miyazawa, Wada and the second author studied the effect of a $2n$-move on Milnor invariants. Here, a $2n$-move is a local move on tangles that inserts or deletes $2n$ half-twists on two parallel strands; in particular, a $2$-move is equivalent to a crossing change.  
The following invariance result shall be useful in Section \ref{sec:classtangles}. It is proved using the same techniques as in \cite[\S 3]{M-W-Y}, and is in some sense a weak version of \cite[Prop.~3.4]{M-W-Y}. 
\begin{lem}\label{lem:8move}
 Let $\sigma$ and $\sigma'$ be two bottom tangles that differ by a single $8$-move. Then $\mu_{L'}(I)\equiv  \mu_{L}(I)\pmod2$ for any sequence $I$ of length at most $4$. 
\end{lem}
\begin{proof}[Sketch of proof]
 The proof follows the very same lines as the proof of \cite[Prop.~3.4]{M-W-Y}. As outlined there, the $i$th longitude of $\sigma'$, for any $i$, is essentially obtained from that of $\sigma$ by inserting the $4$th power of some elements in the free group on the Wirtinger generators of the fundamental group $G(\sigma')$. The result then follows from the fact that for any word $w$ in the free group $F$, we have 
  $$ E(w^4) \equiv 1 + \textrm{(terms of degree $\ge 4$)} \pmod2.$$
 \noindent(This is an elementary property of the Magnus expansion,  which is to be compared with \cite[Claim~3.6]{M-W-Y}.)
\end{proof}

\section{Classifications results for  bottom tangles}\label{sec:classtangles} 

The purpose of this section is to give classification results for bottom tangles, or equivalently for string links, up to the three equivalence relations involved in our main result. 
This will in turn be used in the next two sections to tackle the link case. 

\subsection{Bottom tangles up to clasp-pass moves}

As mentioned in Example~\ref{Ex:Cn}, the equivalence relation on tangles generated by clasp-pass moves, is known to be equivalent to the $C_3$-equivalence defined by Habiro in \cite{Habiro}.
The classification of bottom tangles up to $C_3$-equivalence was given by the first author in \cite[Thm.~4.7]{Meilhan}: 
\begin{theo}\label{key-theorem1}
Two $n$-component bottom tangles are clasp-pass equivalent if and only if they share all invariants $a_2(i)$, $\mu(ij)$, $a_2(ij)$ and $\mu(ijk)$, for any $i,j,k~(i<j<k)$.
\end{theo}

\subsection{Bottom tangles up to band-pass moves}

The equivalence relation generated by band-pass moves, turns out to also be intimately related to Habiro's notion of $C_3$-equivalence. More precisely, we have the following, which will be used implicitly throughout the rest of this paper. 
\begin{lem}\label{band-pass$=C3+c$}
Two tangles are band-pass equivalent if and only if they are $C_3$-concordant.
\end{lem}

\begin{proof}
As mentioned in the introduction, Shibuya showed in \cite{Shibuya} that  
two concordant tangles are related by self-pass moves. 
 Hence two $C_3$-concordant tangles are related by a sequence of self-pass moves and clasp-pass moves. Since these two local moves are realized by a band-pass move,  we have the \lq if' part of the statement. \\
The \lq only if' part is shown using clasper calculus. As noted in Example~\ref{Ex:Cn}, two tangles that differ by a band-pass move, are related by surgery along a $C_1$-tree as shown on the left-hand side of Figure~\ref{band-pass implies C3+c}. 
The first equivalence in the figure follows from Lemma \ref{lem:calculus}~(4). Note that the two leaves denoted by $f$ and $f'$ in the figure cobound a sub-arc of the tangle: we can thus slide, say, $f$ along this sub-arc until it is adjacent to $f'$, and so that the two $C_1$-trees are parallel as shown on the right-hand side of Figure~\ref{band-pass implies C3+c}. 
\begin{figure}[!h]
\includegraphics{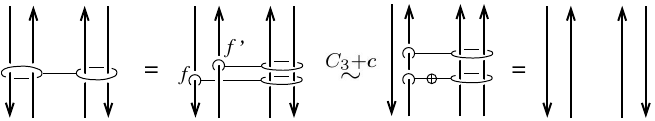} 
\caption{A band-pass move is realized by $C_3$-concordance}\label{band-pass implies C3+c}
\end{figure}

\noindent There are several obstructions to this sliding: we may have to slide the leaf $f$ across some other clasper leaf, and to pass its adjacent edge across some strand of the tangle or some clasper edge. But by Lemma~\ref{lem:calculus}~(1)-(3), such operation introduce new claspers which are all either of degree $\ge 3$ or are doubled $C_2$-trees of antiparallel type. 
According to Lemma \ref{basicmove2},  all such extra claspers can be removed up to $C_3$-concordance.
Note that we may also have to pass $f$ across the other leaf of the doubled $C_1$-tree containing $f$; but this is achieved by a mere isotopy, as illustrated in Figure~\ref{fig:doubledC1}.\footnote{The figure only shows the proof in the untwisted case; the proof in the case of a positive or negative half-twist is strictly similar.}
\begin{figure}[!h]
\includegraphics{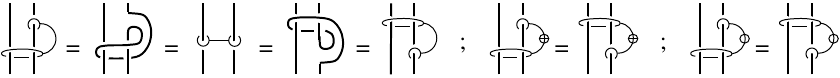} 
\caption{Exchanging the two leaves of a doubled $C_1$-tree is achieved by isotopy. }\label{fig:doubledC1}
\end{figure}

\noindent The result follows, since the two parallel $C_1$-trees shown on the right-hand side of Figure~\ref{band-pass implies C3+c} can be removed by Lemma~\ref{lem:calculus}~(5).
\end{proof}

We deduce the following classification result, which follows from the above lemma, combined with the classification of bottom tangles up to $C_3$-concordance given by the authors in \cite[Thm.~5.8]{MY}. 

\begin{theo}\label{key-theorem2}
Two $n$-component bottom tangles are band-pass equivalent if and only if they share all invariants $a_2(i) \pmod2$, $\mu(ij)$, $\mu(jiij) \pmod2$ and $\mu(ijk)$, for any $i,j,k~(i<j<k)$.
\end{theo}

\subsection{Bottom tangles up to band-$\shp$ moves}

Recall from the introduction that for any pair of distinct indices $i,j$, we set $\varphi(ij)\equiv 4\mu(jiij)+\mu(ij) \pmod8$. 
Our third classification result for bottom tangles reads as follows. 
\begin{theo}\label{key-theorem3}
Two $n$-component bottom tangles are band-$\shp$ equivalent if and only if they share all invariants $\mu(ij) \pmod4$, $\varphi(ij) \pmod8$ and $\mu(ijk) \pmod2$, for any $i,j,k~(i<j<k)$.
\end{theo}

The rest of this section is devoted to the proof of Theorem \ref{key-theorem3}. 

\subsubsection{Proof of the \lq only if\rq~ part of Theorem \ref{key-theorem3}}\label{sec:onlyif}
 
The fact that the mod 4 linking number $\mu(ij)$ is invariant under a band-$\shp$ moves is clear from the definition, hence we only need to prove the invariance of $\varphi(ij)\pmod8$ and $\mu(ijk) \pmod2$. It is enough to
consider the case where $\sigma'$ is obtained from $\sigma$ by a single band-$\shp$ move, 
which is applied on the $s$th and $t$th components of $\sigma$, possibly with $s=t$. 

Suppose first  that $s\neq t$. Without loss of generality we can assume that $\sigma'$ is obtained from $\sigma$ by surgery along the $C_1$-tree illustrated on the left-hand side of Figure \ref{x}.
We prove below the following : 
\begin{equation}%\label{eq:1}
\sigma' \stackrel{C_3+c}{\sim} \left( \sigma\right)_{I\cup Y\cup F'},  \tag{\textasteriskcentered}
\end{equation}
where $I$ is a union of four parallel $C_1$-trees of index $(s,t)$ and $Y$ is a single $C_2$-tree of index $(s,t,t)$ as shown on the right-hand side of Figure \ref{x}, and $F'$ is a union of pairs of parallel $C_2$-trees. 
\begin{figure}[!h]
\includegraphics[scale=0.92]{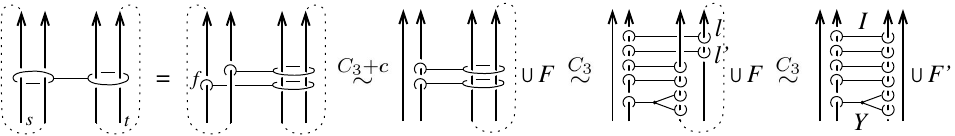} 
\caption{Here $F$ and $F'$ are unions of pairs of parallel $C_2$-trees, that are disjoint from a $3$-ball containing the depicted part. }\label{x}
\end{figure}

\noindent The proof of Equation ($\textasteriskcentered$) is illustrated in Figure \ref{x}. 
The first equivalence in the figure follows from Lemma \ref{lem:calculus}~(4). The second equivalence is given by Lemma \ref{lem:calculus}~(1)-(3) by sliding the leaf labeled $f$ along component $s$, following the orientation; this sliding creates extra terms, namely a union of doubled $C_2$-trees of parallel type, which we may replace with pairs of parallel $C_2$-trees up to $C_3$-concordance using Lemma \ref{basicmove2bis}.
The next equivalence in the figure is given by Lemma \ref{lem:calculus}~(4), followed by Lemma \ref{lem:calculus}~(1)-(3). Finally, we can slide the two leaves labeled by $l$ and $l'$ in the figure along component $t$, following the orientation, until we obtain the union $I\cup Y$ shown on the right-hand side; again, by Lemma \ref{lem:calculus}~(1)-(3) this sliding may introduce further pairs of parallel $C_2$-trees, and Equation ($\textasteriskcentered$) follows. 
We deduce the \lq only if\rq~ part of Theorem \ref{key-theorem3} in the case $s\neq t$, as follows 
(here we pick any  $i,j,k~(i<j<k)$): 
\begin{itemize}
 \item Surgery along $I$ is by definition an $8$-move, so by Lemma \ref{lem:8move} we have that $\mu(jiij)\pmod2$ is preserved. Moreover, this surgery changes  $\mu(ij)$ by $\pm 4 \delta_{\{i,j\},\{s,t\}}$; hence we have that $\varphi(ij)$ changes by $4 \delta_{\{i,j\},\{s,t\}}\pmod8$ under surgery along $I$.\\
Lemma \ref{lem:8move} also directly gives  that $\mu(ijk)\pmod2$ is preserved by surgery along $I$. 
\item  Surgery along $Y$  is a connected sum of a copy of the Whitehead link (Figure~\ref{fig:exclasp}), which is link-homotopic to the unlink. Hence by Lemma~\ref{lem:lh} the invariants $\mu(ij)$ and $\mu(ijk)$ are preserved by surgery along $Y$. 
On the other hand, by Lemma~\ref{keylemma}~(iii), surgery along $Y$ changes $\mu_{\sigma}(jiij)$ by $\pm \delta_{\{i,j\},\{s,t\}}$. 
Hence overall surgery along $I\cup Y$ preserves  $\varphi(ij)\pmod8$.
\item
By Lemma \ref{lem:milnorparallel}, surgery along a pair of parallel $C_2$-trees preserves the invariants $\mu_{\sigma}(jiij)\pmod2$ and $\mu_{\sigma}(ijk)\pmod2$. Since the linking numbers $\mu(ij)$ are invariants of $C_2$-equivalence \cite{Habiro}, we have that $\varphi(ij)\pmod8$ is also invariant.
\end{itemize}

The case $s=t$ is for a large part similar. Indeed, we still have Equation ($\textasteriskcentered$), by a similar argument as the one of Figure \ref{x}. The only difference is that the four depicted strands all belong to the same tangle component, and may be met in various orders when running along this component (in other words, we now ignore the dotted parts in the figure). The deformation of Figure \ref{x} can nonetheless be adapted to each of these cases, using Lemmas \ref{lem:calculus}~(1)-(3) and \ref{basicmove2bis}.\footnote{We may also have to use the equivalence illustrated in  Figure \ref{fig:doubledC1} during the deformation.}
Hence it suffices to discuss how surgery along $I\cup Y\cup F'$, in the case $s=t$, affects the invariants $\mu(ij)$, $\varphi(ij)\pmod8$ and $\mu(ijk)\pmod2$ ($1\leq i<j<k\leq n$): 
\begin{itemize}
 \item Surgery along $I$ is achieved by self-crossing changes, hence leaves the linking numbers $\mu(ij)$ unchanged. 
On the other hand, this surgery being an $8$-move, we still have by Lemma \ref{lem:8move} that $\mu(jiij)\pmod2$ and $\mu(ijk)\pmod2$ are preserved. 
\item  Since $Y$ is a $C_2$-tree of index $(s,s,s)$, it follows from \cite[Thm.~2.1]{FY} that it preserves all Milnor invariant indexed by a sequence $I$ such that each index appears at most twice, hence in particular preserves $\mu(ij)$, $\mu(jiij)$ and $\mu(ijk)$. 
\item
As in the case $s\neq t$ above, Lemma \ref{lem:milnorparallel} ensures that  surgery along $F'$ preserves the invariants $\varphi(ij)\pmod8$ and $\mu_{\sigma}(ijk)\pmod2$. 
\end{itemize}
This concludes the proof of the \lq only if\rq~ part of Theorem \ref{key-theorem3}. 

\subsubsection{Proof of the \lq if\rq~ part of Theorem \ref{key-theorem3}.}

Let $\sigma$ and $\sigma'$ be two $n$-component bottom tangles having same invariants $\mu(ij) \pmod4$, $\varphi(ij) \pmod8$ and $\mu(ijk) \pmod2$, for any $i,j,k~(i<j<k)$.

Since a band-$\shp$ move involving components $i$ and $j$ changes the linking number $\mu(ij)$ by $\pm 4$, and preserves all other linking numbers, we can construct a bottom tangle $\sigma^{(1)}$ which is band-$\shp$ equivalent to $\sigma'$, and such that $\mu_{\sigma}(ij)= \mu_{\sigma^{(1)}}(ij)$ for all indices $i,j$. 
Moreover, recall that the (band-)$\shp$ move is an unknotting operation for knots \cite[Lem.~3.1]{Murakami}. 
Hence we may safely assume that for each $i$, the $i$th component of the bottom tangle $\sigma^{(1)}$  is isotopic to the $i$th component of $\sigma$. In particular, we thus have that  $a_2(\sigma;i)= a_2(\sigma^{(1)};i)$ for each $i$. 
Using the \lq only if\rq~ part of Theorem \ref{key-theorem3}, proved in Section \ref{sec:onlyif}, we also have that 
$\mu_{\sigma}(jiij)\equiv\mu_{\sigma^{(1)}}(jiij)\pmod2$ and 
$\mu_{\sigma}(ijk)\equiv\mu_{\sigma^{(1)}}(ijk)\pmod2$  for any $i,j,k~(i<j<k)$.

The proof of the following lemma is postponed to the end of this section. 
\begin{lem}\label{triple linking}
Let $L$ be an $n$-component bottom tangle; let $\varepsilon\in\{-1,1\}$ and $p,q,r\in \{1,\cdots,n\}~(1\leq p<q<r\leq n)$.   
There exists a bottom tangle $L'$ such that 
\begin{itemize}
\item $L'$ is  band-$\shp$ equivalent to $L$;
\item $\mu_{L'}(pqr)-\mu_{L}(pqr)=2\varepsilon$;
\item $L'$ and $L$ share all invariants $a_2(i) \pmod2$, $\mu(ij)$, $\mu(jiij) \pmod2$ and $\mu(ijk)$, for any 
$i,j,k~(i<j<k)$ and $\{i,j,k\}\neq\{p,q,r\}$.
\end{itemize}
\end{lem}
Using Lemma~\ref{triple linking}, we can construct a bottom tangle $\sigma^{(2)}$, which is  band-$\shp$ equivalent to $\sigma^{(1)}$, hence to $\sigma'$, and such that $\sigma^{(2)}$ and $\sigma$ share all invariants $a_2(i)\pmod2$, $\mu(ij)$, $\mu(jiij) \pmod2$ and $\mu(ijk)$, for any $i,j,k~(i<j<k)$. 

 Theorem~\ref{key-theorem2} then implies that $\sigma^{(2)}$ and $\sigma$ are band-pass equivalent.
Since a band-pass move is a composition of band-$\shp$ moves,  we have that  $\sigma'$ and $\sigma$ are band-$\shp$ equivalent. 
Hence it only remains to prove the above lemma to complete the proof of Theorem~\ref{key-theorem3}.

\begin{proof}[Proof of Lemma \ref{triple linking}]
Consider a pair $P$ of parallel $C_2$-trees for the bottom tangle $L$, with index $(p,q,r)$ as on the left-hand side of Figure \ref{yy}.
The figure illustrates the fact that $L_P$ and $L$ are band-$\shp$ equivalent, as follows. 
The first equivalence is showed by an isotopy and Lemma~\ref{basicmove2bis}, 
and the second equivalence is given by Habiro's move 9. Now, surgery along the $C_1$-tree $C$ shown in the figure, is equivalent to a $\bps$ move, 
which is a local move defined in Appendix~\ref{Appendix}, which differ from the band-pass and band-$\shp$  moves by the relative orientation of two parallel strands. It is shown in Appendix~\ref{Appendix}, that a  $\bps$ move yields band-$\shp$ equivalent tangles. 
The last equivalence in Figure \ref{yy} is then given by clasper surgery and isotopy.
\begin{figure}[!h]
\includegraphics[scale=1.1]{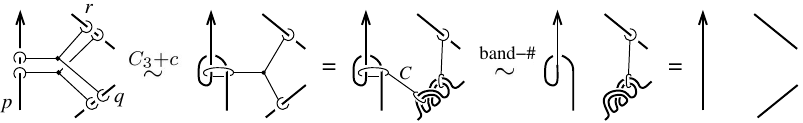} 
\caption{Surgery along a pair of parallel $C_2$-trees, yields band-$\shp$ equivalent tangles.}\label{yy}
\end{figure}
Since the $C_3$-concordance is generated by band-$\shp$ moves, as follows from Lemma \ref{band-pass$=C3+c$} and \cite[Fig.~A.7]{M-N}, we indeed have that $L_P$ and $L$ are band-$\shp$ equivalent.

Let us now consider how surgery along $P$ affects the invariants of the statement. 
On the one hand, by Lemma \ref{keylemma}~(iv) we have that $\mu_{L_P}(pqr)-\mu_{L}(pqr)=\pm 2$, depending on the orientation of the strands, while all other length $3$ Milnor invariants are preserved. 
On the other hand, $L_P$ and $L$ clearly have same linking numbers $\mu(ij)$, since these are $C_2$-equivalence invariants. 
The fact that they also share the invariants $a_2(i) \pmod2$ and $\mu(jiij) \pmod2$, is a direct application of Lemma \ref{lem:milnorparallel}. 
\end{proof}

\section{From bottom tangles to links}\label{sec:Markov}

The purpose of this section is to prove a \lq Markov-type result\rq, namely Theorem \ref{Markov}, which gives necessary and sufficient conditions for two bottom tangles to have clasp-pass, band-pass or band-$\shp$ equivalent closures. 
We begin by setting up some notation. 

For any two indices $i,j$ such that $1\leq i<j\leq n$, we define the $2n$-component string links $\tau_{ij}$, $\tau_{ji}$, 
$\nu_{ij}$ and $\nu_{ji}$ as obtained from the trivial string link by surgery along  the doubled tree claspers represented in  Figure~\ref{2n-string}. 
We also define $\tau_{ij}^{-1}$ (resp. $\tau_{ji}^{-1}$, $\nu_{ij}^{-1}$ and $\nu_{ji}^{-1}$) as obtained by surgery along the clasper of Figure~\ref{2n-string} with a positive half twist inserted on the $\ast$-marked edge. 
\begin{figure}[!h]
\includegraphics[scale=0.8]{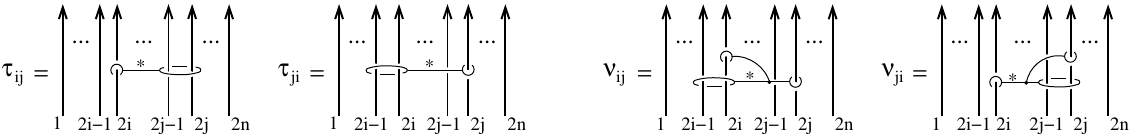} 
  \caption{The string links $\tau_{ij}$, $\tau_{ji}$,  $\nu_{ij}$ and $\nu_{ij}$ ($1\leq i<j\leq n$).}\label{2n-string}
\end{figure}
In what follows, we shall also sometimes use the same notation for the clasper given in Figure~\ref{2n-string}.

Given an $n$-component bottom tangle $\sigma$ and a $2n$-component string link $\tau$, let 
$\sigma\cdot \tau$ denote the bottom tangle obtained by stacking $\sigma$ over $\tau$ 
as illustrated in Figure~\ref{stacking}, equipped with the orientation induced by $\sigma$. 
\begin{figure}[!h]
 \includegraphics[scale=1]{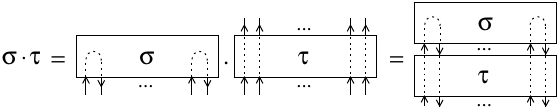}
\caption{Stacking an $n$-component bottom tangle $\sigma$ over a $2n$-component string link $\tau$}\label{stacking}
\end{figure}

\begin{rem}
 This stacking operation corresponds to Habegger-Lin's \emph{left} action of $2n$-component string links on $n$-component string links \cite{HL}, and the \lq Markov-type result\rq~ stated in Theorem \ref{Markov} below for clasp-pass, band-pass and band-$\shp$ equivalence, is in this sense related to \cite[Thm.~2.9 and 2.13]{HL} for link-homotopy. Specifically, surgery along the $C_1$-tree $\tau_{ij}$ corresponds to partial conjugation in \cite[Def.~2.12]{HL}. A full  clasper interpretation of Habegger-Lin's work, in the context of link-homotopy, is carried out by Graff in \cite{Graff}. 
\end{rem}

\begin{lem}\label{lem:nunu}
 For any $i,j~(i<j)$ and for any $n$-component bottom tangle $\sigma$,  we have $\sigma\cdot \nu_{ij}^\varepsilon \stackrel{C_3}{\sim} \sigma\cdot \nu_{ji}^\varepsilon$ ($\varepsilon=\pm 1$).
\end{lem}
\begin{proof}
Let us prove the case $\varepsilon=+1$. 
The proof is illustrated in Figure \ref{fig:nunu}. Starting with $\sigma\cdot \nu_{ij}$, applying Corollary \ref{cor} yields two copies of a simple $C_2$-tree with index $(i,i,j)$. We then use the symmetry of the Whitehead (string) link (Figure~\ref{fig:exclasp}), resulting in two copies of a simple $C_2$-tree with index $(i,j,j)$. Use again Corollary \ref{cor} then yields $\sigma\cdot \nu_{ji}$.
\begin{figure}[!h]
 \includegraphics[scale=0.8]{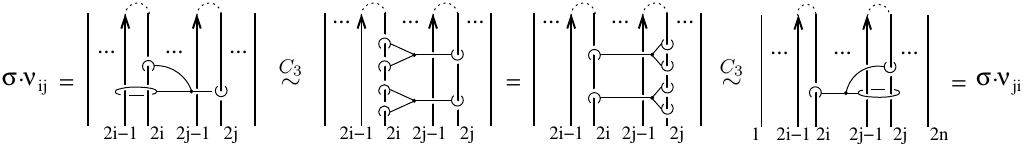}
\caption{The bottom tangles  $\sigma\cdot \nu_{ij}$ and $\sigma\cdot \nu_{ji}$ are $C_3$-equivalent}\label{fig:nunu}
\end{figure}
The case $\varepsilon=-1$ easily follows, using the fact that both $\nu_{ij}\cdot \nu_{ij}^{-1}$ and $\nu_{ji}\cdot \nu_{ji}^{-1}$ are $C_3$-equivalent to the trivial string link. 

\end{proof}

\begin{lem}\label{lem:commute}
 Let $x$ and $y$ be two $2n$-component string links in $\{\tau_{ij}^{\pm 1},\nu_{kl}^{\pm 1}\, ; \, 1\le i\neq j\le n;\, 1\le k<l\le n \}$, and let $\sigma$ be an $n$-component bottom tangle $(n\ge 2)$.
Then the bottom tangles $\sigma\cdot x\cdot y$ and $\sigma\cdot y\cdot x$ are $C_3$-equivalent.
\end{lem}

\begin{proof}
It suffices to prove the result for $x,y\in \{\tau_{ij},\nu_{kl}\, ; \, 1\le i\neq j\le n;\, 1\le k<l\le n \}$, 
by a case by case verification.
 
 In some cases, the result already holds at the string link level. 
In the case where $x=\nu_{kl}$ ($k<l$), we have that $x\cdot y$ and $y\cdot x$ are $C_3$-equivalent by Lemma \ref{lem:calculus}~(1)-(3).
 We also have that $x\cdot y\stackrel{C_3}{\sim}  y\cdot x$ in the case $x=\tau_{ij}$ and $y=\tau_{kl}$ if $i,j,k,l$ are pairwise distinct (note that we might need Lemma \ref{lem:calculus}~(1)).

\noindent  
We now focus on the case where $x=\tau_{ij}$ and $y=\tau_{ji}$ for some $i<j$. 
Consider the union $U$ of doubled $C_1$-trees for the trivial $2n$-component string link $\mathbf{1}$ represented in the center of  Figure~\ref{fig:commute2}.
On the one hand, sliding upwards the leaf labeled $f$ in the figure shows that 
$\mathbf{1}_U\stackrel{C_3}{\sim}\tau_{ij}\cdot \tau_{ji}\cdot \mathbf{1}_D$ by Lemma \ref{lem:calculus}~(1) and (3), where $D$ is the doubled $C_2$-tree shown on the left-hand side of Figure~\ref{fig:commute2}.
On the other hand, by sliding upwards leaf $f'$ we likewise obtain that  
$\mathbf{1}_U\stackrel{C_3}{\sim}\tau_{ji}\cdot \tau_{ij}\cdot \mathbf{1}_{D'}$, where $D'$ is the doubled $C_2$-tree on the right-hand side. 
This, together with Lemma~\ref{basicmove2}~(i), proves that $\sigma\cdot \tau_{ij}\cdot \tau_{ji}$ and $\sigma\cdot \tau_{ji}\cdot \tau_{ij}$ are $C_3$-equivalent for all $i<j$.
\begin{figure}[!h]
 \includegraphics[scale=0.75]{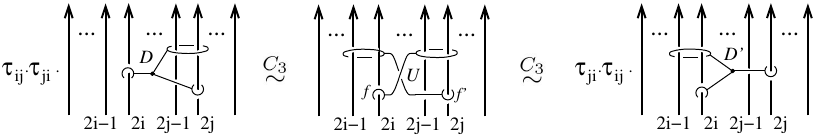}
\caption{ }\label{fig:commute2}
\end{figure}

\noindent All remaining cases involve $\tau_{ij}$ and $\tau_{kl}$ such that exactly $3$ of the  indices $i,j,k,l$ are pairwise distinct. 
Figure \ref{fig:commute3} illustrates the argument for two such cases: applying Lemma~\ref{lem:calculus}~(1)-(3), we can commute the two $C_1$-trees at the cost of an additional doubled $C_2$-tree, and the result follows from  Lemma~\ref{basicmove2}~(i). All other such cases are treated in a similar way, and are left to the reader.  
\begin{figure}[!h]
\includegraphics[scale=0.63]{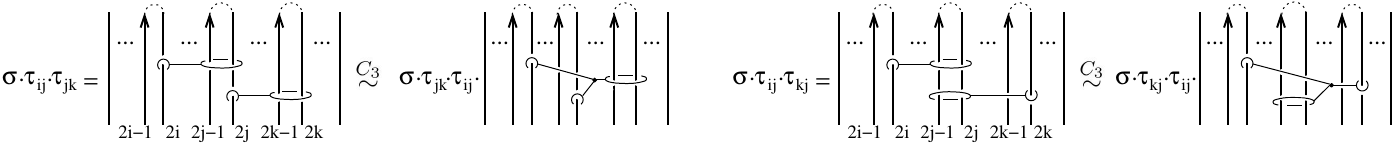}
\caption{ }\label{fig:commute3}
\end{figure}
\end{proof}

\begin{defi}\label{def:omegav}
Let $\sigma$ be an $n$-component bottom tangle. 
Let $\omega=(\omega_{ij})_{i,j}$ and $v=(v_{ij})_{i,j}$ be $n\times n$ matrices with integer entries, 
such that the diagonal entries $\omega_{ii}$ of $\omega$ are 0 and $v$ is upper triangular (that is, $v_{ij}=0\text{ if }i\geq j$).   
We define the $n$-component bottom tangles $\sigma_{\omega}$ and $\sigma_{\omega,v}$ by 
\[\sigma_{\omega}:=\sigma\cdot\prod_{1\leq i\neq j\leq n}\tau_{ij}^{\omega_{ij}}\quad \text{ and }\quad
\sigma_{\omega,v}:=\sigma_{\omega}\cdot\prod_{1\leq i<j\leq n}\nu_{ij}^{v_{ij}},
\]
where for $x\in \{\tau_{ij},\nu_{kl}\, ; \, 1\le i\neq j\le n;\, 1\le k<l\le n \}$,  and for $\alpha\in\mathbb{Z}$, 
$x^{\alpha}$ denotes the stacking product of $|\alpha|$ copies of $x^{\alpha/|\alpha|}$.
\end{defi}
\begin{rem}
 The bottom tangles $\sigma_{\omega}$ and $\sigma_{\omega,v}$ are in general not well-defined from the above formula, since we have not specified any order on the entries of $n\times n$ integer matrices.
 However, in view of Lemma~\ref{lem:commute}, their $C_3$-equivalences classes are well-defined. Since we shall only consider these tangles up to $C_3$-equivalence in what follows, we may safely use Definition~\ref{def:omegav}.
\end{rem}

Let us denote by $\Omega_n$ the set of $n\times n$ matrices with integer entries
whose diagonal entries are all $0$. 
We denote by 
\[T:=\{(v_{ij})_{i,j}\in \Omega_n~|~v_{ij}=0\text{ for }i\geq j\}\]
the set of strictly upper triangular integer matrices, and we further set 
\[ \Delta:=\{(\delta_{ij})_{i,j}\in \Omega_n~|~\delta_{ij}\in\{0,1\}~\forall i,j\}.\]
 
The proof of the following result is postponed to the end of this section. 
\begin{prop}\label{keyprop}
Let $\sigma$ and $\sigma'$ be $n$-component bottom tangles, whose closures $\widehat{\sigma}$ and 
$\widehat{\sigma'}$ are equivalent. We have the following: 
\begin{enumerate}
\item[(1)] There exist $\omega \in \Omega_n$ and $v\in T$ such that
$\sigma'$ is clasp-pass equivalent to $\sigma_{w,v}$.
\item[(2)] There exists a matrix $\omega \in \Omega_n$  such that
$\sigma'$ is band-pass equivalent to $\sigma_{\omega}$.
\item[(3)] There exists a matrix $\delta \in \Delta$  such that
$\sigma'$ is band-$\shp$ equivalent to $\sigma_{\delta}$.
\end{enumerate}
\end{prop}

Assuming Proposition~\ref{keyprop} for now, we obtain the following \lq Markov-type result\rq~for the three equivalence relations studied in this paper.
\begin{theo}\label{Markov}
Let $L$ and $L'$ be two $n$-component links, and let $\sigma$ and $\sigma'$ be 
bottom tangles whose closures $\widehat{\sigma}$ and $\widehat{\sigma'}$  are equivalent to $L$ and $L'$, respectively. We have the following: 
\begin{enumerate}
\item[(1)] $L$ and $L'$ are clasp-pass equivalent if and only if 
there exist matrices $\omega \in \Omega_n$ and $v\in T$ such that
$\sigma'$ is clasp-pass equivalent to $\sigma_{w,v}$.
\item[(2)] $L$ and $L'$ are band-pass equivalent if and only if 
there exists a matrix $\omega \in \Omega_n$  such that
$\sigma'$ is band-pass equivalent to $\sigma_{\omega}$.
\item[(3)] $L$ and $L'$ are band-$\shp$ equivalent if and only if 
there exists a matrix $\delta \in \Delta$  such that
$\sigma'$ is band-$\shp$ equivalent to $\sigma_{\delta}$.
\end{enumerate}
\end{theo}
\begin{proof}
We only prove here the first statement: the proof of the other two statements are strictly similar and make use of Proposition~\ref{keyprop} in the exact same way. \\
The \lq if\rq~part of the statement is clear. Indeed, for any $\omega \in \Omega_n$ and $v\in T$, we have that  $\widehat{\sigma_{\omega,v}}=\widehat{\sigma_{\omega}}=\widehat{\sigma}$ because 
all claspers in Figure~\ref{2n-string} become trivial under the closure operation by Habiro's move 1. \\
For the \lq only if\rq~part of the statement, thanks to Lemma \ref{lem:commute} 
it is enough to show the case where $\widehat{\sigma'}$ is obtained from $\widehat{\sigma}$ by a single clasp-pass move. Then there exists a bottom tangle $\sigma''$ which is clasp-pass equivalent to $\sigma'$ and such that $\widehat{\sigma''}=\widehat{\sigma}$ (this is because we can always find \lq closing arcs\rq~for $\sigma''$ that are disjoint from a $3$-ball supporting the clasp-pass move.) 
By Proposition~\ref{keyprop}~(1), we have that there exist matrices $\omega \in \Omega_n$ and $v\in T$ such that $\sigma''$, hence $\sigma'$, is clasp-pass equivalent to $\sigma_{w,v}$. 
\end{proof}
It only remains to prove Proposition~\ref{keyprop}.

\begin{proof}[Proof of Proposition~\ref{keyprop}~(1)]
We regard the bottom tangle $\sigma'$ as a band sum 
 of $\widehat{\sigma'}$ and the trivial bottom tangle, along a union $b$ of \lq trivially embedded\rq~ bands as illustrated on the left-hand side of Figure~\ref{bandsum}. 
\begin{figure}[!h]
\includegraphics[scale=1.1]{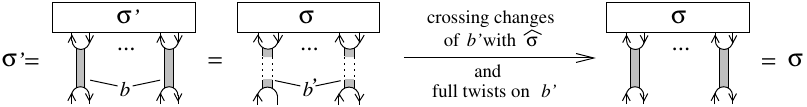} 
  \caption{}\label{bandsum}
\end{figure}

\noindent Performing an isotopy turning $\widehat{\sigma'}$ into $\widehat{\sigma}$, and further isotopies of the attaching regions of the bands, we obtain a decomposition of $\sigma'$ as a band sum of $\widehat{\sigma}$ and the trivial bottom tangle. This band sum is along the image $b'$ of $b$ under these isotopies; these bands $b'$ are thus knotted and linked together and with $\widehat{\sigma}$. 
Therefore $\sigma$ can be obtained from this band sum by crossing changes 
between the bands $b'$ and strands of $\widehat{\sigma}$, and adding full-twists to $b'$, see Figure~\ref{bandsum}. 
Note that a crossing change between bands of $b'$ can be realized by the former deformations. 
By Lemma \ref{lem:commute}, it is enough to consider each of these deformations of the bands $b'$ independently. 

We note that adding a full twist to the $i$th band, is achieved by surgery along a union of doubled $C_1$-trees  
$\tau_{ij}^{\pm1}~(j\neq i)$, see Figure~\ref{deform2}.\footnote{The figure only illustrates the case of a negative full twist, but the positive case is similar.} 
\begin{figure}[!h]
 \includegraphics[scale=1]{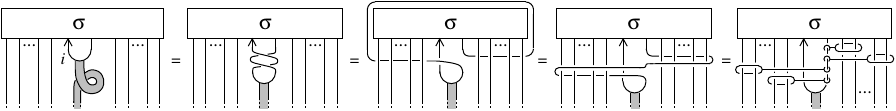}
  \caption{Performing full twists on bands by surgery along claspers $\tau_{ij}^{\pm 1}$}\label{deform2}
\end{figure}

\noindent Now, a crossing change between the $i$th component of $b'$ and the $j$th component of $\widehat{\sigma}$ (possibly $i=j$), is achieved by surgery along a doubled $C_1$-tree $C$ whose doubled leaf intersects $b'$ transversely (see the left-hand sides of Figures~\ref{deform3} and \ref{deform3bis}). 
We may regard $\widehat{\sigma}$ as obtained from the trivial link by surgery along a union $U$ of simple $C_1$-trees (Example~\ref{Ex:Cn}). 
Let us distinguish two cases. 
 \\
 The case where $i\neq j$ is summarized in Figure~\ref{deform3}.
  We aim at dragging the simple leaf $f$ of $C$ towards the band $b'$, so as to obtain a copy $C'$ of $\tau_{ij}^{\pm 1}$ or $\tau_{ji}^{\pm 1}$ (see the right-hand side of the figure). 
 This is achieved by passing successively the leaf $f$ through some leaves of $U$, and passing the edge of $C$ across some tangle strand: by  Lemma~\ref{lem:calculus}~(1)-(3), each such operation can be achieved, up to $C_3$-equivalence, at the cost of an additional doubled $C_2$-trees, and we only have to show that up to $C_3$-equivalence, all such extra terms are either copies of $\nu_{ij}$ with $i<j$ or can be deleted. As Figure~\ref{deform3} illustrates,\footnote{The figure shows the case of a leaf exchange, but the  operation passing the edge of $C$ across a tangle strand is treated similarly.} there are two subcases:
   \begin{itemize}
 \item if $i=k$, the leaf exchange introduces a copy of $\nu_{ij}$ up to $C_3$-equivalence; if $i>j$, we replace $\nu_{ij}$ with $\nu_{ji}$ using Lemma \ref{lem:nunu}; 
 \item if $i\neq k$, the extra $C_2$-tree can be deleted up to $C_3$-equivalence by Lemma~\ref{basicmove2}~(i).
 \end{itemize}
\begin{figure}[!h]
 \includegraphics[scale=1.1]{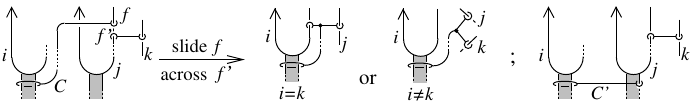}
  \caption{Doubled $C_1$-tree realizing a crossing change: case $i\neq j$}\label{deform3}
\end{figure}
 
 \noindent The case where $i=j$ is illustrated in Figure~\ref{deform3bis}.
 Dragging the simple leaf $f$ of $C$ yields the doubled $C_1$-tree $C'$ shown on the right-hand side, 
 and is equivalent to a union of $\tau_{ij}^{\pm 1}~(j\in\{1,...,n\}\setminus\{i\})$ 
 (see the left-hand side of Figure~\ref{fig:doubledC1} and Figure~\ref{deform2}).
 As in the previous case, we analyse the additional doubled $C_2$-trees that may appear in this dragging process: 
 \begin{itemize}
 \item if $i=k$, then the extra term given by a leaf exchange is a doubled $C_2$-tree with index $(i^{(2)},i,i)$ as shown in the figure, and by Lemma~\ref{basicmove2}~(ii) it can be deleted up to $C_3$-equivalence;
 \item if $i\neq k$, the leaf exchange introduces a copy of $\nu_{ik}$ up to $C_3$-equivalence, and we conclude as in the preceding case. 
 \end{itemize}
\begin{figure}[!h]
 \includegraphics[scale=1.1]{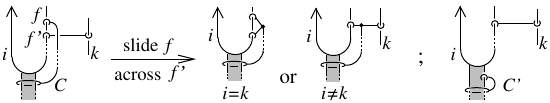}
  \caption{Doubled $C_1$-tree realizing a crossing change: case $i=j$}\label{deform3bis}
\end{figure}

Hence we have shown that $\sigma'$ can be deformed into $\sigma$ at the cost of stacking copies of $\tau_{ij}^{\pm 1}$ ($i\neq j$) and $\nu_{ij}^{\pm 1}$ ($i<j$), and the proof is complete. 
\end{proof}

\begin{proof}[Proof of Proposition~\ref{keyprop}~(2)]
Using  Proposition~\ref{keyprop}~(1), there exists matrices $\omega\in\Omega_n$   and $v\in T$ such that $\sigma'$ is clasp-pass, hence band-pass equivalent to $\sigma_{\omega,v}$. It then follows directly from  Lemmas~\ref{basicmove2}~(iii) and \ref{band-pass$=C3+c$} that $\sigma_{\omega,v}$ is band-pass equivalent to $\sigma_{\omega}$. 
\end{proof}

\begin{proof}[Proof of Proposition~\ref{keyprop}~(3)]
By Proposition~\ref{keyprop}~(2), there exists a matrix $\omega\in\Omega_n$ 
such that $\sigma'$ is band-pass equivalent to $\sigma_{\omega}$. 
It only remains to observe that  $\tau^{\pm2}_{ij}~(i\neq j)$ is band-$\shp$ equivalent to the trivial string link, as shown in Figure~\ref{band sharp cancel tau2}.  
\begin{figure}[!h]
\includegraphics[width=8.5cm]{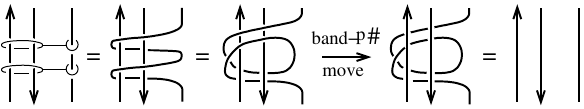} 
  \caption{A band-$\shp$ move realizes $\tau_{ij}^2$}\label{band sharp cancel tau2}
\end{figure}
\end{proof}

\section{Proof of Theorems 1 and 2}

In  order to prove our main results, we shall need the following technical result. 
\begin{prop}\label{trace}
Let $\sigma$ be an $n$-component bottom tangle.  
Let $\omega=(\omega_{ij})_{i,j}\in\Omega_n$ and $v=(v_{ij})_{i,j}\in T$. 
For any $i,j,k~(i<j)$, we have the following: 
\begin{enumerate}
\item[(1)] $a_2(\sigma_{\omega,v};i)=a_2(\sigma;i)$ and $\mu_{\sigma_{\omega,v}}(ij)=\mu_{\sigma}(ij)$.\\[-.5em]
\item[(2)]  $a_2(\sigma_{\omega,v};ij)-a_2(\sigma_{\omega};ij)=2v_{ij}$ and $\mu_{\sigma_{\omega,v}}(ijk)=\mu_{\sigma_{\omega}}(ijk)$.\\[-.5em]
\item[(3)] 
$a_2(\sigma_{\omega};ij)-a_2(\sigma;ij) = - \mu_\sigma(ij)(\omega_{ij}+\omega_{ji})$, \\
$\mu_{\sigma_{\omega}}(jiij)-\mu_{\sigma}(jiij)\equiv 
\mu_{\sigma}(ij)(\omega_{ij}+\omega_{ji}) \pmod2$ 
and \\
$\mu_{\sigma_{\omega}}(ijk)-\mu_{\sigma}(ijk)
=\mu_{\sigma}(ij)(\omega_{ki}-\omega_{kj})+\mu_{\sigma}(jk)(\omega_{ij}-\omega_{ik})
+\mu_{\sigma}(ki)(\omega_{jk}-\omega_{ji}).$
\end{enumerate}
\end{prop}
\begin{proof}
Part (1) of the statement is clear. 
On one hand,  all four claspers depicted in Figure~\ref{2n-string} become trivial when taking the closure of the $i$th component of $\sigma_{\omega,v}$ and deleting all other components, so that the resulting knot is isotopic to the closure of the  $i$th component of $\sigma$, hence has same invariant $a_2$.
On the other hand, it is readily checked that surgery along all four claspers depicted in 
Figure~\ref{2n-string} preserves the linking numbers. 

Part (2) addresses the variation of $a_2(ij)$ and $\mu(ijk)$ under stacking $\sigma_\omega$ over the string link $\prod_{1\le i<j\le n} \nu_{ij}^{v_{ij}}$, that is, under surgery along the doubled $C_2$-trees $\nu_{st}$ ($s<t$). 
Since $\1_{\nu_{st}}$ is link-homotopic to $\1$, we immediately have that $\mu_{\sigma_{\omega,v}}(ijk)=\mu_{\sigma_{\omega}}(ijk)$ by Lemma~\ref{lem:lh}.
Moreover, using Corollary \ref{cor}, the fact that   $a_2(\sigma_{\omega,v};ij)-a_2(\sigma_{\omega};ij)=2v_{ij}$ follows directly from Lemma~\ref{keylemma}~(ii).

It thus remains to prove part (3), which investigates the effect of surgery along the $C_1$-trees $\tau_{ij}$ and 
$\tau_{ji}$ ($i<j$).
Recall from \cite{M-N,Matveev} that the $\Delta$-equivalence of (string) links, which is known to coincide with the $C_2$-equivalence (Example~\ref{Ex:Cn}), is classified by the linking numbers. This implies that $\sigma$ is $C_2$-equivalent to $\tau_{(1)} = \prod_{1\le i<j\le n}  L_{ij}^{\mu_\sigma(ij)}$, 
where $L_{ij}^{\pm 1}$ is the $2n$-component string link  represented on the left-hand side of  Figure \ref{fig:LT}, and where the product is taken using the lexicographic order. 
(Note in particular that any $C_1$-tree intersecting a single component can be deleted up to $C_2$-equivalence, using calculus of claspers \cite[\S~4.1]{Habiro}.)
Hence by Lemma~\ref{lem:calculus} we have 
$$ \sigma\stackrel{C_3}{\sim} \sigma_{(2)}\cdot \tau_{(1)}, $$
where $\sigma_{(2)}$ is an $n$-component bottom tangle obtained from the trivial one by surgery along $C_2$-trees. 
\begin{figure}[!h]
\includegraphics[scale=0.9]{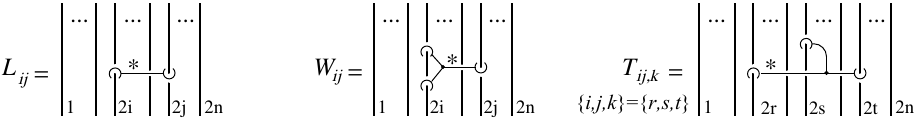} 
  \caption{The string links $L_{ij}$, $W_{ij}$ and $T_{ij,k}$ ($i<j$). Adding a positive half-twist on the $\ast$-marked edges defines the string links  $L_{ij}^{-1}$, $W_{ij}^{-1}$ and $T_{ij,k}^{-1}$.}\label{fig:LT}
\end{figure}

Figure \ref{fig:toto} illustrates the effect of stacking $\sigma$ over $\tau_{ij}$. 
\begin{figure}[!h]
\includegraphics[scale=0.8]{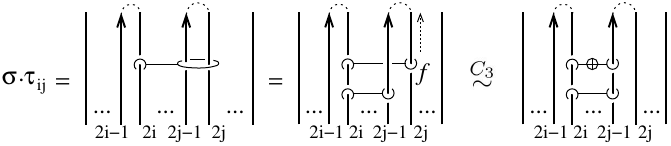} 
  \caption{}\label{fig:toto}
\end{figure} 
The first equivalence is given by Lemma \ref{lem:calculus}~(4). The idea is then to slide the leaf marked $f$ against the orientation of $\sigma$, until we obtain two parallel $C_1$-trees that differ by a twist as shown in the figure: Lemma~\ref{lem:calculus}~(5) then ensures that these two $C_1$-trees cancel. Sliding $f$ in this way, involves passing it across all leaves in $\sigma_{(2)}\cdot \tau_{(1)}$ intersecting the $j$th component; using Lemma~\ref{lem:calculus}~(1)-(3) repeatedly, this yields the equivalence
$$ \sigma\cdot \tau_{ij}\stackrel{C_3}{\sim} \sigma\cdot W_{ij}^{-\mu_\sigma(ij)}\prod_{k\neq i,j} T_{ij,k}^{\mu_\sigma(jk)}, $$
where $W_{ij}^{\pm 1}$ and $T_{ij,k}^{\pm 1}$ are the $2n$-component string link represented in Figure \ref{fig:LT}. 
The following variation formulas then directly follow from Lemma~\ref{keylemma}. 
\begin{equation*}%\label{eq:variation}
 \begin{split}
  a_{2}(\sigma\cdot \tau_{ij};rs) - a_{2}(\sigma;rs) = -\mu_\sigma(ij)\delta_{\{r,s\},\{i,j\}}; \\ 
  \mu_{\sigma\cdot \tau_{ij}}(srrs)-\mu_{\sigma}(srrs) \equiv \mu_\sigma(ij)\delta_{\{r,s\},\{i,j\}}  \pmod2; \\ 
  \mu_{\sigma\cdot \tau_{ij}}(rst) -\mu_{\sigma}(rst)  =\mu_\sigma(jk)\textrm{sgn}
  \left(\begin{smallmatrix} r & s & t \\
                                         i & j & k
\end{smallmatrix}\right),%\delta_{\{r,s,t\},\{i,j,k\}}, 
 \end{split}
\end{equation*}
where sgn$\left(\begin{smallmatrix} r & s & t \\
                                         i & j & k
\end{smallmatrix}\right)$ denotes the signature of the permutation $\left(\begin{smallmatrix} r & s & t \\
                                         i & j & k
\end{smallmatrix}\right)$ if $\delta_{\{r,s,t\},\{i,j,k\}}=1$, and is by convention zero otherwise.

Figure~\ref{fig:tototo} likewise illustrates the effect of stacking $\sigma$ over $\tau_{ji}$. 
\begin{figure}[!h]
\includegraphics[scale=0.8]{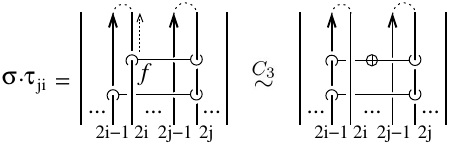} 
  \caption{}\label{fig:tototo}
\end{figure} 

A similar sliding process as discussed above, shows that 
$$ \sigma\cdot \tau_{ji}\stackrel{C_3}{\sim} \sigma\cdot W_{ij}^{-\mu_\sigma(ij)} \cdot \prod_{k\neq i,j} T_{ij,k}^{-\mu_\sigma(ik)},
$$
(note that the equivalence uses the symmetry of $W_{ij}$, illustrated in  Figure \ref{fig:exclasp}). 
Using Lemma~\ref{keylemma}, these equivalences yield variation formulas for the invariants $a_2$, $\mu(srrs)\pmod2$ and $\mu(rst)$ that are similar to the above ones. 
Combining these leads to the desired result. 
\end{proof}
\medskip

The clasp-pass classification of links, Theorem~1, follows from the combination of Theorem~\ref{key-theorem1}, Theorem~\ref{Markov}~(1) and Proposition~\ref{trace}.
For clarity, let us outline the proof of the \lq if\rq\, part of the statement. 

Let $\sigma$ and $\sigma'$ be $n$-component bottom tangles that share all invariants $a_2(i)~(1\le i\le n)$ and $\mu(ij)~(1\le i<j\le n)$, and suppose that there exists integers 
$x_{rs}\in \mathbb{Z}$ $(1\le r\neq s\le n)$ such that for all $i,j,k$ $(1\leq i<j<k\leq n)$ we have 
$$ a_2(\sigma';ij)-a_2(\sigma;ij)\equiv \mu_{\sigma}(ij)(x_{ij}+x_{ji})\pmod2  $$
$$\mu_{\sigma'}(ijk)-\mu_{\sigma}(ijk)=\mu_{\sigma}(ij)(x_{ki}-x_{kj})+\mu_{\sigma}(jk)(x_{ij}-x_{ik})
+\mu_{\sigma}(ki)(x_{jk}-x_{ji}).
$$
This in particular implies that there exists, for all $i,j$ with $1\leq i<j\leq n$, some $y_{ij}\in \mathbb{Z}$ such that 
$$ a_2(\sigma';ij)-a_2(\sigma;ij) = -\mu_{\sigma}(ij)(x_{ij}+x_{ji}) + 2y_{ij}.  $$
We can thus consider the square matrices $\omega=(x_{ij})_{i,j}\in \Omega_n$ (setting $x_{ii}=0$ for all $i$), and 
$v=(y_{ij})_{i,j}\in T$ (where $y_{ij}=0$ for all $i,j$ such that $i\ge j$), so that Proposition~\ref{trace} applies.
This result implies that all invariants $a_2(i)$, $\mu(ij)$, $a_2(ij)$ and $\mu(ijk)$ agree for $\sigma_{\omega,v}$ and $\sigma'$, so 
theses two bottom tangles are clasp-pass equivalent by Theorem~\ref{key-theorem1}.
The fact that the closures of $\sigma$ and $\sigma'$ are clasp-pass equivalent then follows from Theorem~\ref{Markov}~(1).
\medskip
 
The classification of links up to band-pass and band-$\shp$ moves (Theorem~2) is obtained in the same way from Theorems~\ref{key-theorem2}, \ref{key-theorem3}, \ref{Markov}~(2), (3) and Proposition~\ref{trace}~(1),(3).

\appendix
\section{Links and bottom tangles up to $\bps$ moves}\label{Appendix}

\subsection{The $\ps$- and $\bps$-moves}
 A natural variant of the pass- and $\shp$-moves, is the \emph{$\ps$-move} illustrated in the center of Figure \ref{fig:gamm}. 
\begin{figure}[!h]
\includegraphics[scale=0.85]{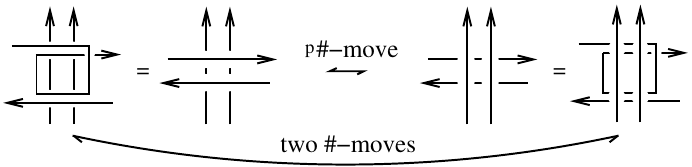}
  \caption{A $\ps$-move is realized by two $\shp$-moves.}\label{fig:gamm}
\end{figure} 
 The figure shows how, using the same argument as Murakami and Nakanishi in \cite[Fig.~A.7]{M-N}, a $\ps$-move can be realized by two applications of the $\shp$-move. The same trick shows that a pass-move is achieved by two $\ps$-moves. In fact, we have the following simple fact. 
 \begin{prop}\label{lem:ps}
 The equivalence relation generated by the $\ps$-move, coincides with the pass-equivalence.
 \end{prop}
 \begin{proof}
 This readily follows from the above observation and the figure below, which illustrates the fact that a $\ps$-moves can be realized by two pass-moves.
$$ \textrm{\includegraphics[scale=0.85]{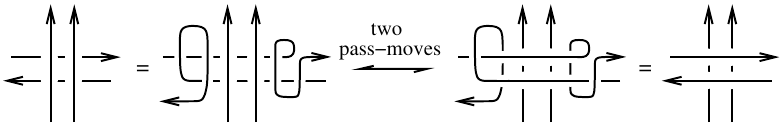}} $$
 \end{proof}
 
 A \emph{$\bps$ move} is naturally defined as a $\ps$-move such that each pair of parallel strands are from the same components.  
 
 The observation in Figure~\ref{fig:gamm} shows that the $\bps$ move defines an intermediate equivalence relation between the band-pass and the band-$\shp$ equivalence. 
However, the proof of Proposition~\ref{lem:ps} no longer applies to the case of a $\bps$ move. 
 As a matter of fact, the classification results stated below show that the $\bps$ equivalence  
is different from the band-pass equivalence or the band-$\shp$ equivalence.
\begin{rem}\label{rem:saito}
Since both the band-pass move and the $\bps$ move are generated by band-$\shp$ moves, the latter is equivalent to the \lq unoriented pass move\rq\, introduced by Saito in \cite{Saito}, which is the unoriented version of the local moves shown in Figure \ref{pass}, where parallel strands are assumed to belong to the same link component. 
\end{rem}

\subsection{Classification results}
This section briefly outlines how the techniques and results of this paper can be simply adapted to the 
$\bps$ equivalence. 

The key observation is that surgery along a pair of parallel $C_2$-trees yields $\bps$ equivalent tangles. This is apparent from Figure~\ref{yy}, which summarizes the proof of Lemma~\ref{triple linking}. 
In particular, this means that the technical Lemma~\ref{triple linking} still holds, when replacing the 
band-$\shp$ equivalence with the $\bps$ equivalence.

Since the $\bps$ move is achieved by two band-$\shp$ moves (Figure~\ref{fig:gamm}), it preserves
$a_2(i) \pmod2$. It also preserves $\mu(jiij) \pmod2$ since
$\varphi(ij)$ is a band-$\shp$ equivalence invariant (Theorem~\ref{key-theorem3}), and
$\mu(ij)$ is a $\bps$-equivalence invariant.  
Using the above-mentioned refined version of Lemma~\ref{triple linking} for 
$\bps$ equivalence, the proof of the \lq if' part of Theorem \ref{key-theorem3} can be adapted.
Hence we have the following.

\begin{theo}\label{key-theorem4}
Two $n$-component bottom tangles are $\bps$-equivalent if and only if they share all invariants $a_2(i) \pmod2$, $\mu(ij)$, $\mu(jiij) \pmod2$ and $\mu(ijk) \pmod2$, for any $i,j,k~(i<j<k)$.
\end{theo}

Since $\tau_{ij}^2$ is $\bps$ equivalent to the trivial string link (see
Figure~\ref{band sharp cancel tau2}), Proposition~\ref{keyprop}~(3) and
Theorem~\ref{Markov}~(3) still hold for the $\bps$ equivalence.
Hence the proof of Theorem~2~(2) can be straightforwardly adapted to establish the following classification theorem.

\begin{theo}\label{main-theorem_last}
Let $L$ and $L'$ be $n$-component links, and let $\sigma$ and $\sigma'$ be
bottom tangles whose closures $\widehat{\sigma}$ and $\widehat{\sigma'}$
are equivalent to $L$ and $L'$ respectively.
The links $L$ and $L'$ are $\bps$-equivalent if and only if $\sigma$ and $\sigma'$
share all invariants $a_2(i)\pmod2~(i\in\{1,...,n\})$ and $\mu(ij)~(1\le i<j\le n)$ and
the following system has a solution over $\mathbb{Z}$.
\[
\left\{
\begin{array}{l}
\mu_{\sigma'}(jiij)-\mu_{\sigma}(jiij)\equiv \mu_{\sigma}(ij)(x_{ij}+x_{ji}) \pmod2\qquad\qquad\qquad\,\,\,\,\, (1\leq i<j \leq n)\\
\mu_{\sigma'}(ijk)-\mu_{\sigma}(ijk)\equiv \mu_{\sigma}(ij)(x_{ki}\!-\!x_{kj})+\mu_{\sigma}(jk)(x_{ij}\!-\!x_{ik})
+\mu_{\sigma}(ki)(x_{jk}\!-\!x_{ji})\pmod2 \\
 \qquad \qquad\qquad\qquad\qquad\qquad\qquad\qquad\qquad\qquad\qquad\qquad\qquad\qquad (1\leq i<j<k\leq n)
\end{array}
\right.
\]
\end{theo}

\end{document}